\documentclass[a4paper,10pt,reqno]{amsart}
\usepackage[foot]{amsaddr}
\usepackage{amsmath}
\usepackage{amsthm,enumerate, esint}
\usepackage{mathtools}
\usepackage{hyperref,url}
\usepackage{graphicx}
\usepackage{amssymb}

\usepackage{appendix}

\usepackage[english]{babel}

\usepackage{fancyhdr}

\usepackage{calc}
\usepackage{url}

\usepackage[text={6in,8.6in},centering]{geometry}

\usepackage{srcltx}

\usepackage[T1]{fontenc}

\usepackage[usenames,dvipsnames]{color}

\theoremstyle{plain}
\newtheorem{theorem}{Theorem}[section]
\theoremstyle{plain}
\newtheorem{lemma}{Lemma}[section]
\newtheorem{proposition}{Proposition}[section]

\newtheorem{definition}{Definition}[section]
\newtheorem{remark}{Remark}[section]

\renewcommand{\(}{\left(}
\renewcommand{\)}{\right)}
\renewcommand{\[}{\left[}
\renewcommand{\]}{\right]}
\newcommand{\eps}{\varepsilon}

\newcommand{\To}{\longrightarrow}
\newcommand{\be} {\begin{equation}}
\newcommand{\ee} {\end{equation}}
\newcommand{\bea} {\begin{eqnarray}}
\newcommand{\eea} {\end{eqnarray}}
\newcommand{\Bea} {\begin{eqnarray*}}
	\newcommand{\Eea} {\end{eqnarray*}}

\newcommand{\ba} {\beta}

\newcommand{\ga} {\gamma}
\newcommand{\Ga} {\Gamma}
\newcommand{\Om} {\Omega}

\newcommand{\De} {\Delta}
\newcommand{\la} {\lambda}

\newcommand{\Si} {\Sigma}

\newcommand{\no} {\nonumber}

\newcommand{\lab} {\label}
\newcommand{\va} {\varphi}
\newcommand{\var} {\varepsilon}

\newcommand{\R}{\mathbb R}
\newcommand{\N}{\mathbb N}
\newcommand{\Rn}{\mathbb R^N}

\newcommand{\deb}{\rightharpoonup}
\newcommand{\Hs}{\dot{H}^s(\mathbb{R}^{N})}
\newcommand{\IKtf}{\bar I_{K,t,f}^{\gamma}}
\newcommand{\IKto}{\bar I_{K,t,0}^{\gamma}}
\newcommand{\dI}{(\bar I_{K,t,f}^{\gamma}){'}}
\newcommand{\dIo}{(\bar I_{K,t,0}^{\gamma})'}

\newcommand{\hms}{(\dot{H}^{s})'}
\newcommand{\w}{w_{\ga,t}^{\tau}}
\newcommand{\x}{\bar{x}_0}
\catcode`\@=11

\newcommand{\authorfootnotes}{\renewcommand\thefootnote{\@fnsymbol\c@footnote}}%
\allowdisplaybreaks

\def\N{{I\!\!N}}

\numberwithin{equation}{section} \allowdisplaybreaks

\begin{document}
\title[Fractional Hardy-Sobolev equations]{Fractional Hardy-Sobolev equations\\ with nonhomogeneous terms}

\date{}

\author[Mousomi Bhakta]{Mousomi Bhakta\textsuperscript{1}}
\address{\textsuperscript{1}Department of Mathematics, Indian Institute of Science Education and Research, Dr. Homi Bhaba Road, Pune-411008, India}
\email{mousomi@iiserpune.ac.in}

\author[Souptik Chakraborty]{Souptik Chakraborty\textsuperscript{1}}
\address{\textsuperscript{1}Department of Mathematics, Indian Institute of Science Education and Research, Dr. Homi Bhaba Road, Pune-411008, India}
\email{souptik.chakraborty@students.iiserpune.ac.in}

\author[Patrizia Pucci]{Patrizia Pucci\textsuperscript{2}}
\address{\textsuperscript{2}Dipartimento di Matematica e Informatica, Universit\`a degli Studi di Perugia --
Via Vanvitelli 1, I-06123 Perugia, Italy}
\email{patrizia.pucci@unipg.it}

\keywords{Nonlocal equations,  fractional Laplacian, Hardy-Sobolev equations,  profile decomposition, Palais-Smale decomposition, energy estimate, positive solutions, min-max method.}

\maketitle

\begin{abstract}
This paper deals with existence and multiplicity of positive solutions to the following class of nonlocal equations with critical nonlinearity:
\begin{equation*}\begin{cases}
(-\Delta)^s u -\gamma\dfrac{u}{|x|^{2s}}=K(x)\dfrac{|u|^{2^*_s(t)-2}u}{|x|^t}+f(x) \quad\mbox{in}\quad\Rn,\\
 \qquad\qquad\qquad\quad u\in \dot{H}^s(\Rn),
\end{cases}
\end{equation*}
where $N>2s$, $s\in(0,1)$, $0\leq t<2s<N$ and $2^*_s(t):=\frac{2(N-t)}{N-2s}$. Here $0 <\gamma<\gamma_{N,s}$ and $\gamma_{N,s}$ is the best Hardy constant in the fractional Hardy inequality. The coefficient $K$ is a positive
continuous function on $\mathbb{R}^{N}$, with $K(0)=1=\lim_{|x|\to\infty}K(x)$. The perturbation $f$ is a nonnegative nontrivial functional in the dual space $\dot{H}^s(\Rn)$
of $\dot{H}^s(\Rn)$ i.e., $\prescript{}{(\dot{H}^{s})'}{\langle}f,u{\rangle}_{\dot{H}^s}\geq 0$, whenever $u$ is a nonnegative function  in $\dot{H}^s(\Rn)$. We establish the profile decomposition of the Palais-Smale sequence associated with the functional. Further, if $K\geq 1$ and $\|f\|_{(\dot{H}^s)'}$ is small enough (but $f\not\equiv 0$), we establish existence of at least two positive solutions to the above equation.

\medskip

\noindent
\emph{\bf 2010 MSC:} 35R11,  35A15, 35B33, 35J60
\end{abstract}

\section{introduction}
The paper deals with the following fractional Hardy-Sobolev equation with nonhomogeneous term
\begin{equation}\tag{$E^{\ga}_{K, t, f}$}
\label{P}
(-\Delta)^s u -\gamma\frac{u}{|x|^{2s}}=K(x)\frac{|u|^{2^*_s(t)-2}u}{|x|^t}+f(x) \quad\mbox{in}\quad\Rn,\quad
 u\in \dot{H}^s(\Rn),
\end{equation}
where $N>2s$, $s\in(0,1)$, $0\leq t<2s<N$ and $2^*_s(t):=\frac{2(N-t)}{N-2s}$.   Clearly, $2<2^*_s(t)\leq \frac{2N}{N-2s}=2^*_s$.  Here $0 <\gamma<\gamma_{N,s}$, where $\gamma_{N,s}$ is the best Hardy constant in the fractional Hardy inequality
$$\gamma_{N,s} \int_{\Rn} \frac{|u(x)|^2}{|x|^{2s}} \,{\rm d}x \leq \int_{\Rn}|\xi|^{2s}|\mathcal{F}(u)(\xi)|^2 \,{\rm d}\xi,\quad \gamma_{N,s}=2^{2s}\frac{\Gamma^2(\frac{N+2s}{4})}{\Gamma^2(\frac{N-2s}{4})}.$$ Throughout the paper $\mathcal{F}(u)$ denotes the Fourier transform of $u$.  Moreover, $$\lim_{s\to 1}\gamma_{N,s}:= \bigg(\frac{N-2}{2}\bigg)^2,$$
which is exactly the  best Hardy constant in the classical case $s=1$.
The symbol $(-\De)^s$ denotes the  fractional Laplace operator which can be defined for any function $u$ of the Schwartz class functions $\mathcal{S}(\Rn)$  as follows:
\begin{equation*}
  \left(-\Delta\right)^su(x): = C_{N,s} \mbox{P.V.} \int_{\Rn}\frac{u(x)-u(y)}{|x-y|^{N+2s}} \, dy, \quad C_{N,s}= \frac{4^{s}\Ga(N/2+ s)}{\pi^{N/2}|\Ga(-s)|}.
\end{equation*}
For the sharp Hardy inequalities in general fractional Sobolev spaces
$W^{s,p}(\Rn)$, $1<p<\infty$,
as well as for historical comments in the case $p=2$, we refer the interested reader to~\cite{FS} and the references therein. While for fractional Hardy-Sobolev-Maz'ya inequality, we mention the recent contribution \cite{Ma} and for fractional Hardy inequality in Heisenberg group we refer to \cite{AM}.
Throughout the paper the homogeneous fractional Sobolev space is denoted by
$$\dot{H}^s(R^{N}): =\bigg\{u\in L^{2^*_s}(\R^N) \; : \; \iint_{\mathbb{R}^{2N}}\frac{|u(x)-u(y)|^2}{|x-y|^{N+2s}}\,{\rm d}x\,{\rm d}y<\infty\bigg\},$$
and it is endowed 	with the inner product
$\langle\cdot,\cdot\rangle_{\dot{H}^s}$ and corresponding Gagliardo norm
	$$\|u\|_{\dot{H}^{s}(\R{^N})}:=\left(\frac{C_{N,s}}{2} \iint_{\mathbb{R}^{2N}} \frac{|u(x)-u(y)|^2}{|x-y|^{N+2s}}\,{\rm d}x\,{\rm d}y\right)^{1/2}=\bigg( \int_{\Rn}|\xi|^{2s}|\mathcal{F}(u)(\xi)|^2 \, {\rm d}\xi\bigg)^\frac{1}{2}.$$
In literature there are several definitions of the fractional Laplacian in which different
normalizing constants $C_{N,s}$ appear. The constant $C_{N,s}$ is chosen so that the above definition  is equivalent with the one via the Fourier transform,
which is called {\em classical}.
The definition via Fourier transform recovers the standard
Laplacian
as $s\to1$,  which however cannot be represented by other nonlocal formulas.
	
	In \eqref{P}, the functions $K$ and $f$ satisfy the properties.
\begin{enumerate}
\item[${\bf (K)}$] {\it $0<K\in C(\Rn)$, $K(0)=1=\lim_{|x|\to\infty}K(x)$. }
\end{enumerate}
\begin{enumerate}
\item[${\bf (F)}$] {\it $f\not\equiv 0$ is a nonnegative functional in the dual space $\dot{H}^s(\Rn)'$ of $\dot{H}^s(\Rn)$}, i.e. whenever $u$ is a nonnegative function  in $\dot{H}^s(\Rn)$ then $\prescript{}{(\dot{H}^{s})'}{\langle}f,u{\rangle}_{\dot{H}^s}\geq 0$.
\end{enumerate}
Using the Hardy inequality, it is easy to see that the operator $L_{\ga, s}:=(-\De)^s-\frac{\ga}{|x|^{2s}}$ with $0\leq \ga<\ga_{N,s}$ is a positive operator. The request $\ga<\ga_{N,s}$ is fairly natural since we are looking for positive solutions.
In this case the Hardy-Sobolev inequality holds for $L_{\ga, s}$, which states that if $0\leq t<2s<N$, then
\be\lab{S_gats}
S_{\ga, t, s}=S_{\ga, t, s}(\Rn):=\inf_{u\in\Hs\setminus\{0\}}\frac{\tfrac{C_{N,s}}{2}\displaystyle\iint_{\R^{2N}}\frac{|u(x) - u(y)|^2}{|x-y|^{N+2s}}\,{\rm d}x\,{\rm d}y -\gamma\int_{\R^N} \frac{|u|^2}{|x|^{2s}} \, {\rm d}x}{\bigg(\displaystyle\int_{\Rn}\frac{|u|^{2^*_s(t)}}{|x|^t}\,{\rm d}x\bigg)^\frac{2}{2^*_s(t)}}
\ee
is finite, strictly positive and achieved (see \cite{GRSZ,GS}). Observe that thanks to \cite{GRSZ}, any minimizer for \eqref{S_gats} leads (up to a constant) to a nonnegative variational solution of the
\begin{equation}\tag{$E^{\ga}_{1,t,0}$}\label{Q}
(-\Delta)^s u -\gamma\frac{u}{|x|^{2s}}=\frac{|u|^{2^*_s(t)-2}u}{|x|^t},\quad u\in \dot{H}^s(\Rn).
\end{equation}
If $\ga=0=t$, then $S_{\ga, t, s}$ reduces to the best Sobolev constant $S_{0,0,s}=S$ which is known to be achieved by $C_{N,s}(1+|x|^2)^{-\frac{N-2s}{2}}$ and any minimizer of $S$ leads (up to a constant) to a nonnegative solution of equation $(E^{0}_{1,0,0})$ i.e., \eqref{Q} with $\ga=0=t$.

\begin{definition}
{\rm	We say $u \in \Hs$ is a {\em positive weak solution of} \eqref{P} if $u>0$ in $\Rn$ and for every $\phi\in\Hs$, we have
\begin{align*}
	\frac{C_{N,s}}{2}\iint_{\R^{2N}}\frac{(u(x) - u(y))(\phi(x)-\phi(y))}{|x-y|^{N+2s}} {\rm d}x\,{\rm d}y &-\gamma\int_{\R^N} \frac{u\phi}{|x|^{2s}}   {\rm d}x\\
&= \int_{\R{^N}}K(x)\frac{|u|^{2_s^*(t)-2}u\phi}{|x|^t}   {\rm d}x+ \prescript{}{\hms}{\langle}f,\phi{\rangle}_{\dot{H}^s},
\end{align*}
	where $\prescript{}{\hms}{\langle}.,.{\rangle}_{H^s}$ denotes the duality bracket between $\Hs$ and its dual $\Hs'$.}
\end{definition}
\medskip

\begin{remark}\lab{r:norm}
{\rm	For $0<\ga<\ga_{N,s},$
	\be
\|u\|_{\ga}:=\bigg(\frac{C_{N,s}}{2}\iint_{\R^{2N}}\frac{|u(x)-u(y)|^2}{|x-y|^{N+2s}}\;{\rm d}x{\rm d}y-\ga\int_{\Rn}\frac{|u|^2}{|x|^{2s}}\;{\rm d}x\bigg)^{\tfrac{1}{2}}\no
	\ee
	defines a norm in $\Hs$ which is equivalent to the standard norm in $\Hs$. In particular,
	$$\sqrt{1-\frac{\ga}{\ga_{N,s}}}\|u\|_{\dot{H}^s}\le \|u\|_{\ga}
\le\|u\|_{\dot{H}^s}.$$
The corresponding equivalent inner product $\langle \cdot,\cdot\rangle_{\ga}$ in the fractional homogeneous Hilbert space $\Hs$ is given by
	\be
	\langle u,v\rangle_{\ga}:=\frac{C_{N,s}}{2}\iint_{\R^{2N}}\frac{\big(u(x)-u(y)\big)\big(v(x)-v(y)\big)}{|x-y|^{N+2s}}\;{\rm d}x{\rm d}y-\ga\int_{\Rn}\frac{uv}{|x|^{2s}}\;{\rm d}x.\no
	\ee
Finally, for simplicity we endow in what follows the weighted Lebesgue space $L^{2^*_s(t)}(\Rn, |x|^{-t})$
with the norm $\|u\|_{L^{2^*_s(t)}(\Rn, |x|^{-t}))}=\left(
\int_{\Rn}\frac{|u|^{2_s^*(t)}}{|x|^t}\;{\rm d}x\right)^{1/2_s^*(t)}$.}
\end{remark}

We are going to prove existence and multiplicity of positive solutions of \eqref{P} in the spirit of \cite{BCG,BP}.
Under the conditions on $K$ and $f$ stated above, equation
\eqref{P} can be regarded as a perturbation problem of the homogeneous equation
\eqref{Q}. It is known from \cite{GS} that when $0<\ga<\ga_{N,s}$ or $\{\ga=0 \mbox{ and } 0<t<2s\}$, then any nonnegative minimizer for $S_{\ga, t, s}$ is positive, radially symmetric, radially decreasing, and approaches zero as $|x|\to\infty$. The main question to be addressed is whether positive solution can survive after a perturbation of type \eqref{P} or not.

For $\ga=0=t$, this kind of question was recently studied by the first and third author of the current paper in \cite{BP}. For Schr\"{o}dinger operator (without Hardy term), same type of questions were addressed in \cite{BCG}. However for $\ga\neq 0$ the presence of the Hardy potential requires a new argument
to dealt with. One of the key steps to prove the multiplicity result is a careful analysis of the Palais-Smale level. Theorem~\ref{th:PS} studies the profile decomposition of any Palais-Smale sequence possessed by the
underlying functional associated to~\eqref{P}. We show that concentration takes place along a single profile when
$t>0$, while concentration takes place along two different profiles when $t=0$. In the local case $s=1$, $t=0$ and $f=0$ Smets deals with the profile decomposition  in~\cite{Sm}. In bounded domains and again in the local case $s=1$,
paper~\cite{BS} treats the case of  all $t\geq 0$. However, extension of the latter results in the nonlocal case $s\in (0,1)$ and in the entire space $\Rn$ is highly nontrivial and requires several delicate estimates and techniques to deal with.
\medskip


In local case $s=1$, we refer \cite{FP, Sm}, where authors have studied the local version of $(E^\ga_{K,0,0})$ in $\Rn$. In the nonlocal case, when the domain is a bounded subset of $\Rn$, existence of positive solutions of $\eqref{P}$ in $\Om$  with $\ga=0=t$ (i.e., without Hardy and Hardy-Sobolev terms) and Dirichlet boundary condition has been proved in \cite{SZY}. Existence of sign changing solutions of
$$(-\De)^su=|u|^\frac{4s}{N-2s}u+\eps f \mbox{ in }\ \Omega, \quad u=0\ \mbox{ in }\, \Rn\setminus\Om,$$  where $f\geq 0, f\in L^\infty(\Om)$ has been studied in \cite{AT} and existence of two positive solutions have been established in \cite{WZ} when $f$ is a continuous function with compact support in $\Om$. In the nonlocal case, when the domain is the entire space $\Rn$, but $\ga=0$, we refer 
to~\cite{BCG, BP}, where multiplicity of positive solutions have been studied in presence of a nonhomogeneous term. 

There is a wide literature regarding problems involving the fractional Hardy potential.  Avoiding to disclose the discussion we refer to the following (far from being complete) list of works and references therein \cite{AADP, AMPP, BBGM, GGJP, DMPS, FLS, GS}. In \cite{DMPS} Dipierro, et al. study the equation $(E^\ga_{1,0,0})$ (i.e., \eqref{Q} with $t=0$)
and prove existence of a ground state solution, qualitative properties of positive solutions and asymptotic behavior of solutions  at
both $0$ and  infinity. In \cite{BBGM}, the authors deal with the Green function for $L_{\ga,s}$ ($0<\ga<\ga_{N,s}$) and show when the integral representation of the weak solution is valid.

It is worth noting that solutions of \eqref{Q} do not belong to $L^\infty(\Rn)$ as soon as $\ga>0$, because of the singularity at zero. In fact solutions blow up at origin (see \cite{DMPS, GS}). For this reason, it seems more difficult to handle \eqref{P} in the general case using the fine analysis of blow up technique quoted above.
\medskip

To the best of our knowledge, so far there has been no papers in the literature, where existence and multiplicity of positive solutions of Hardy-Sobolev type equations (with $\ga\neq 0$ and $t\geq 0$) in $\Rn$, have  been established in the nonhomogeneous case $f\neq 0$. Also the profile decomposition in the nonlocal case with the Hardy term is completely new and the proof is very involved, delicate and complicated compared with the local case $s=1$. The proofs are not at all an easy adoption of the local case or the case $\ga=0$. The multiplicity results in this paper is new even in the local case $s=1$, but we leave the obvious changes, when $s=1$, to the interested reader.
\vspace{2mm}

\noindent Below we state the main result.

\begin{theorem}\lab{th:ex-f}
Assume that ${\bf (F)}$ and ${\bf (K)}$ are satisfied, with
$K\geq 1$ in $\mathbb R^N$.
If
$$ \|f\|_{\hms}< C_t\sqrt{1-\frac{\ga}{\ga_{N,s}}}S_{\ga,t,s}^{\tfrac{N-t}{4s-2t}}, \quad\mbox{where}\quad C_t=\bigg(\frac{4s-2t}{N-2t+2s}\bigg)\bigg((2_s^*(t)-1)\|K\|_{L^{\infty}}\bigg)^{-\big(\frac{N-2s}{4s-2t}\big)},$$
then

(i) For $t>0$, equation \eqref{P} admits two positive solutions;

(ii) For $t=0$, equation \eqref{P} admits a positive solution. In addition, if
$\|K\|_{L^\infty}<\big(\frac{S}{S_{\ga,0,s}}\big)^\frac{N}{N-2s}$
then \eqref{P} admits two positive solutions.
\end{theorem}
\vspace{2mm}

\begin{remark}
{\rm It is worth mentioning that   ${S}>
{S_{\ga,0,s}}$ for any $\ga>0$. To see this, we denote by $W$ the unique positive solution of \eqref{W} and
let $W_{\ga,0}$ be a minimum energy positive solution (ground state solution) of \eqref{Q} with $t=0$. Then,
$$I_{1,0,0}^{\ga}(W_{\ga,0})\leq I_{1,0,0}^{\ga}(W)<I_{1,0,0}^{0}(W).$$
A straight forward computation yields that
$I_{1,0,0}^{0}(W)=\frac{s}{N}S^\frac{N}{2s}$ and $I_{1,0,0}^{\ga}(W_{\ga,0})=\frac{s}{N}S_{\ga,0,s}^\frac{N}{2s}$. Consequently, $S>S_{\ga,0,s}$ for any $\ga>0$.  From this observation, it immediately follows that if $K\equiv 1$, then $(E^\ga_{1,t,f})$ admits two positive solutions for all $t\geq 0$ under the given assumtion
${\bf (F)}$ on $f$.}
\end{remark}

Note that the Hardy-Sobolev embedding
$\dot{H}^s(\Rn) \hookrightarrow L^{2^*_s(t)}(\Rn, |x|^{-t})$ for any $0\leq t<2s$ is continuous, but not compact. This noncompactness of the embedding even locally in any neighbourhood of zero leads to other additional difficulties, and more importantly, to new phenomenon concerning the possibility of blow up. Thus the variational functional associated
to~\eqref{P} does not satisfy the Palais-Smale  condition, briefly called $(PS)$
condition.  The lack of compactness of
the functional associated to~\eqref{P} is due to a concentration phenomenon. We analyze this noncompactness in Theorem~\ref{th:PS}, which is one of the most important theorems of the paper.
Using this theorem we prove existence and multiplicity of positive solutions to \eqref{P} in Theorem \ref{th:ex-f}. For that first we decompose $\dot{H}^s(\Rn)$ into three components which are homeomorphic to the interior, boundary and the exterior of the unit ball in $\dot{H}^s(\Rn)$ respectively. Then we prove that the energy functional associated to \eqref{P} attains its infimum on one of the components which serves as our first positive solution.  The second positive solution is obtained via a careful analysis on the $(PS)$ sequences associated to the energy functional and we construct a min--max critical level $\kappa_t$, where the $(PS)$ condition holds.

\medskip

This paper has been organised in the following way.
In Section~\ref{sec2}, we prove the Palais-Smale decomposition theorem associated with the functional corresponding to \eqref{P} (see Theorem~\ref{th:PS}). In Section~\ref{sec3}, we show existence of two positive solutions of \eqref{P}, namely
Theorem~\ref{th:ex-f}.  Appendix~A  contains some basic estimates which are used in proving the
Palais-Smale characterization theorem in Section~\ref{sec2}.
\vspace{2mm}

	{\bf Notation:} In this paper $\dot{H}^s(\Rn)'$ (or in short $\hms$) denotes the dual space of $\Hs$, $C, C', C'', C''',\cdots$ denote the generic constant which may vary from line to line. The symbol $B_r(y)$ stands for the ball centered at $y\in\Rn$ and of radius $r$. For simplicity $B_r$ means $B_r(0)$. Moreover, $u_+:=\max\{u, 0\}$ and $u_-:=-\min\{u, 0\}$. Therefore, according to our notation $u=u_+-u_-$. Finally, $S$ is the best Sobolev constant.

\section{Palais-Smale decomposition}\label{sec2}
 In this section we study the Palais-Smale sequences (in short, $(PS)$ sequences) of the functional $\bar I_{K,t,f}^{\gamma}$ associated to \eqref{P}
 \begin{align}\label{EF}
 \bar I_{K,t,f}^{\gamma}(u)&:= \frac{C_{N,s}}{4}\iint_{\R^{2N}}\frac{|u(x) - u(y)|^2}{|x-y|^{N+2s}}\,{\rm d}x\,{\rm d}y -\frac{\gamma}{2} \int_{\R^N}\frac{|u|^2}{|x|^{2s}} \,{\rm d}x \no\\
 &\quad\qquad- \frac{1}{2_s^*(t)}\int_{\R^N}K(x)\frac{|u|^{2_s^*(t)}}{|x|^t}\,{\rm d}x-\prescript{}{\hms}{\langle}f,u{\rangle}_{\dot{H}^s}\\
 &=\frac{1}{2}\|u\|_{\gamma}^{2}-\frac{1}{2_s^*(t)}\int_{\R^N}K(x)\frac{|u|^{2_s^*(t)}}{|x|^t}\,{\rm d}x -\prescript{}{\hms}{\langle}f,u{\rangle}_{\dot{H}^s},\nonumber
 \end{align}
 where $K$ and $f$ satisfy $(\bf{K})$ and $(\bf{F})$ respectively.

 We say that the sequence $(u_n)_n\in \Hs$ is a $(PS)$ sequence for $\bar I_{K,t,f}$ at level $\ba$ if $\bar I_{K,t,f}(u_n)\to \ba$ and $(\bar I_{K,t,f})'(u_n)\to 0$ in $\hms$. It is easy to see that the weak limit of a $(PS)$ sequence solves \eqref{P}  except the positivity.

 However the main difficulty is that the $(PS)$ sequence may not converge strongly and hence the weak limit can be zero even if $\ba>0.$
 The main purpose of this section is to classify $(PS)$ sequences of
 the functional $\bar I_{K,t, f}^{\gamma}$. Classification of $(PS)$
 sequences has been done for various problems having lack of compactness, to quote a few, we cite \cite{BP, PS, PS-2} in the nonlocal case with $\ga=0=t$,
 while in the local case \cite{BS, Sm}  with Hardy potentials and in~\cite{St}   without Hardy potentials. We also refer to~\cite{TF} for a more abstract approach  of the profile decomposition in general Hilbert spaces.  We establish
 a classification theorem for the $(PS)$ sequences of \eqref{EF} in the spirit of the above results. In \cite{BP, PS}, the noncompactness is completely described by the single blow up profile $W$, which is a solution of
\be\tag{$E_{1,0,0}^{0}$}\lab{W}(-\Delta)^s W=|W|^{2_s^*-2}W\quad\mbox{in}\quad\Rn, \quad W\in \dot{H}^s(\Rn).\ee
In \cite{BS, Sm} (the local case $s=1$), the noncompactness are due to  concentration occurring through two different profiles. Possibility of two different type of profiles are still present for \eqref{P} in the case $t=0$.

Let $t=0$ and let $W$ be any solution of \eqref{W}. Then,
 it can be easily verified that any sequence of the form
\be\lab{Wrn}
W^{r_n,\;y_n}(x):=K(y)^{-\frac{N-2s}{4s}}r_n^{-\frac{N-2s}{4s}}W\big(\frac{x-y_n}{r_n}\big),\ee
 is a $(PS)$ sequence for $\bar I_{K,0,0}^{\ga}$ if $y_n\to y\neq 0$ and $r_n\to 0$. If $y=0$, then $W^{r_n,\;y_n}$ remains a $(PS)$ sequence for $\bar I_{K,0,0}^{\ga}$ provided that $\frac{|y_n|}{r_n}\to\infty$. Also  $W^{r_n,\;y_n}\deb 0$ in $\Hs$ by \cite[Lemma~3]{PS}.
 \vspace{2mm}

Further, let $W_{\ga,t}$ be any solution of \eqref{Q} (where $t\geq 0$). Define a sequence $(W_{\gamma,t}^{R_n,0})_n$ of the form
\be\lab{Wgrn}
W_{\gamma,t}^{R_n,0}(x):=R_n^{-\frac{N-2s}{4s}}W_{\ga,t}\big(\frac{x}{R_n}\big),\ee
 where $R_n\to 0$. Then $W_{\gamma,t}^{R_n,0}\deb 0$ in $\Hs$ and
$(W_{\gamma,t}^{R_n,0})_n$ is a $(PS)$ sequence for $\bar I_{K,t,0}^{\ga}$ for $t\geq 0$.

\begin{theorem}\label{th:PS}
	Let $(u_n)_n$ be a $(PS)$ sequence for $\bar I_{K,t,f}^{\gamma}$ at the level $\ba.$ Then up to a subsequence, still denoted by $(u_n)_n$, the next properties hold.

If $t=0$, then there exist $n_1,\;n_2\in \N$, $n_2$ sequences $(R_n^k)_n\subset \R^{+}\;(1\leq k\leq n_2)$, $n_1$ sequences $(r_n^j)_n\subset \R^{+}$ and $(y_n^j)_n\subset \Rn\setminus \{0\}\;(1\leq j\leq n_1)$ and $0\leq \bar u\in\Hs$ such that
	\bea
	&(i)&\; u_n=\bar u+\sum_{j=1}^{n_1}K(y^j)^{-\frac{N-2s}{4s}}(W^j)^{r_n^j,\;y_n^j}+\sum_{k=1}^{n_2}(W_{\gamma,t}^k)^{R_n^k,0}+o(1)\no\\
	&(ii)&\; \dI (\bar u)= 0\no\\
	&(iii)&\; R_n^k\to 0 \;(1\leq k\leq n_2) \,\,\mbox{and}\,\, r_n^j\to 0\,(1\leq j\leq n_1) \no\\	
	&(iv)&\mbox{either }  \; y_n^j\to y^j\in \Rn \mbox{or } |y^j|\to\infty \mbox{ and }\frac{r^j_n}{|y_n^j|}\to 0\;(1\leq j\leq n_1)\no\\
	&(v)&\; \ba = \IKtf (\bar u)+\sum_{j=1}^{n_1}K(y^j)^{-\frac{N-2s}{2s}}\bar I_{1,0,0}^{0}(W^j)+\sum_{k=1}^{n_2}\bar I_{1,t,0}^{\gamma}(W_{\gamma,t}^{k})+o(1)\no\\
	&(vi)&\; \bigg|\log\big(\frac{r^i_n}{r_n^j}\big)\bigg|+\bigg|\frac{y_n^i-y_n^j}{r_n^j}\bigg|\underset{n\to\infty}\To\infty \quad\mbox{for}\,\, i\neq j\no\\
	&(vii)&\; \bigg|\log\big(\frac{R^k_n}{R_n^l}\big)\bigg|\underset{n\to\infty}\To\infty \quad\mbox{for}\,\, k\neq l\no,
	\eea
	where $o(1)\to 0$  in $\Hs$ as $n\to\infty$,  $(W^j)^{r_n^j,\;y_n^j}$ and $(W_{\gamma,t}^k)^{R_n^k,0}$ are $(PS)$ sequences of the form \eqref{Wrn} and \eqref{Wgrn} respectively, with $W=W^j$ and $W_{\ga,t}=W_{\ga,t}^k$.
	
		When $t>0$, the same conclusions hold, with $W^j= 0$ for all $j$.
	
In the case $n_1=0,\; n_2=0$ the above properties $(i)$--$(vii)$ are valid without $W,\;W_{\gamma},\;R_n^k,\;r_n^j$. In addition, if $u_n\geq 0,$ then $\bar{u}\geq 0$ and $W^j\geq 0$ for all $1\leq j \leq n_1$, $W^k_{\ga,t}\geq 0$ for all $1\leq k \leq n_2$. Therefore, $W^j=W_0$ for all $1\leq j\leq n_1$ due to the uniqueness up to the translation and dilation for the positive solutions of \eqref{W}.
\end{theorem}

\begin{proof}
	We prove the theorem in several steps.
	
\noindent\underline{\bf Step 1:} Using standard arguments it follows that there exists $M>0$ such that $$\|u_n\|_{\ga}<M \quad\mbox{for all } n\in\mathbb{N}.$$
 More precisely, as $n\to\infty$
	\begin{align*}
	\ba+o(1)+o(1)\|u_n\|_{\gamma} &\geq \IKtf(u_n) \, - \,
	\frac{1}{2_s^*(t)} \prescript{}{\hms}{\big\langle}\dI(u_n), u_n{\big\rangle}_{\dot H^s}\\
	 & =\bigg(\frac{1}{2}-\frac{1}{2_s^*(t)}\bigg)\|u_n\|_{\gamma}^{2} - \left(1- \frac{1}{2_s^*(t)} \right) \prescript{}{\hms}{\langle}f, u_n{\rangle}_{\dot H^s}\\
	 &\geq\bigg(\frac{1}{2}-\frac{1}{2_s^*(t)}\bigg)\|u_n\|_{\ga}^{2}-  \left(1- \frac{1}{2_s^*(t)} \right)\|f\|_{\hms}\|u_n\|_{\dot{H}^s}.
	\end{align*}
	As $2_s^*(t)>2$, from the above estimate it follows that $(u_n)_n$ is bounded in $\Hs$. Consequently, there exists $\bar u$ in $\Hs$
such that, up to a subsequence, still denoted by $(u_n)_n$, $u_n\rightharpoonup \bar u$ in $\Hs$
and $u_n\to \bar u$ a.e. in $\mathbb R^N$.  Moreover, as  $\prescript{}{\hms}{\big\langle}\dI(u_n), v{\big\rangle}_{\dot H^s}\rightarrow 0$ as $k\rightarrow\infty$ for all  $v\in\Hs$, then
	\be\label{B6}
	(-\De)^su_n -\gamma \frac{u_n}{|x|^{2s}} - K(x)|u_n|^{2_s^*(t)-2}u_n -f\To 0\quad \mbox{in}\quad \dot H^s(\Rn)'.
	\ee
	\underline{\bf Step 2:}
	From \eqref{B6}, letting $n \rightarrow \infty$, we get
	\begin{equation}\lab{24-5-1}
	\langle u_n, v\rangle_{\gamma}-\int_{\Rn}K(x)\frac{|u_n|^{2_s^*(t)-2}u_nv}{|x|^t}\;{\rm d}x- \, \prescript{}{\hms}{\langle}f, v{\rangle}_{\dot{H}^s} \, {\rightarrow} \,  0.
\end{equation}
As $u_n\rightharpoonup \bar u$ in $\Hs$, it is easy to see that $\langle u_n, v\rangle_{\gamma}\to \langle \bar u, v\rangle_{\gamma}$ for all $v\in\Hs$.

 \noindent
 {\bf Claim 1}: $\displaystyle\int_{\Rn}K(x)\frac{|u_n|^{2_s^*(t)-2}u_nv}{|x|^t}\;{\rm d}x \To \int_{\Rn}K(x)\frac{|\bar u|^{2_s^*(t)-2}\bar u v}{|x|^t} {\rm d}x$
 for all $v\in\Hs$.

Indeed, $u_n\to \bar u$ a.e. in $\mathbb R^N$ and
	\be\lab{24-5-2}\int_{\Rn}K(x)\frac{|u_n|^{2_s^*(t)-2}u_nv}{|x|^t}\;{\rm d}x = \int_{B_R}K(x)\frac{|u_n|^{2_s^*(t)-2}u_nv}{|x|^t}\;{\rm d}x +\int_{\Rn\setminus B_R}K(x)\frac{|u_n|^{2_s^*(t)-2}u_nv}{|x|^t}\;{\rm d}x. \ee
	 On $B_R$ we will show the convergence using Vitali's convergence theorem. For that, given any $\varepsilon>0$, we choose $\Om\subset
B_R$ such that
	$\displaystyle
\bigg(\int_{\Omega}\frac{|v|^{2_s^*(t)}}{|x|^t}{\rm d}x\bigg)^{\tfrac{1}{2_s^*(t)}}<\frac{\varepsilon}{\|K\|_{L^{\infty}}(MS_{\ga,t,s}^{-\frac{1}{2}})^{2_s^*(t)-1}}$. Since $\frac{|v|^{2_s^*(t)}}{|x|^t}$ is in  $L^1(\Rn)$, the above choice makes sense. Therefore,
\bea
	 \bigg{|}\int_{\Omega}K(x)\frac{|u_n|^{2_s^*(t)-2}u_nv}{|x|^t}\;{\rm d}x\bigg{|} &\leq& \|K\|_{L^{\infty}(\Rn)} \int_{\Omega}\frac{|u_n|^{2_s^*(t)-1}|v|}{|x|^t}\;{\rm d}x \no\\
	 &\leq& \|K\|_{L^{\infty}(\Rn)}\bigg(\int_{\Omega}\frac{|u_n|^{2_s^*(t)}}{|x|^t}{\rm d}x\bigg)^{\tfrac{2_s^*(t)-1}{2_s^*(t)}}\bigg(\int_{\Omega}\frac{|v|^{2_s^*(t)}}{|x|^t}{\rm d}x\bigg)^{\tfrac{1}{2_s^*(t)}}\no\\
	 &\leq& \|K\|_{L^{\infty}(\Rn)}S_{\ga,t,s}^{-\frac{2^*(t)-1}{2}}\|u_n\|_{\gamma}^{2_s^*(t)-1}\bigg(\int_{\Omega}\frac{|v|^{2_s^*(t)}}{|x|^t}{\rm d}x\bigg)^{\tfrac{1}{2_s^*(t)}}< \varepsilon\no
	 \eea
	 Thus $K\frac{|u_n|^{2_s^*(t)-2}u_nv}{|x|^t}$ is uniformly integrable in $B_R$. Therefore,  using Vitali's convergence theorem, we can pass the limit in the 1st integral on RHS of \eqref{24-5-2}.
	
	 To estimate the integral now on $B_R^c$, we first set $v_n=u_n-\bar u$.
	 Then $v_n\rightharpoonup 0$ in $\Hs.$ It is not difficult to see that for every $\varepsilon >0$ there exists
$C_\varepsilon>0$ such that
	 $$\bigg{|}|v_n+\bar u|^{2_s^*(t)-2}(v_n+\bar u)-|\bar u|^{2_s^*(t)-2}\bar u \bigg{|} <\varepsilon |v_n|^{2^*_s(t)-1}+C_\varepsilon|\bar u|^{2_s^*(t)-1}.$$
Therefore,
\begin{align*}
&\bigg{|}\int_{B_R^c}K(x)\bigg\{\frac{|u_n|^{2_s^*(t)-2}u_n}{|x|^t}-\frac{|\bar u|^{2_s^*(t)-2}\bar u}{|x|^t}\bigg\}v\;{\rm d}x\bigg{|}\\
	 	 &\leq \|K\|_{L^{\infty}(\Rn)}\bigg[\varepsilon \int_{B_R^c}\frac{|v_n|^{2_s^*(t)-1}|v|}{|x|^t}\;{\rm d}x+C_\varepsilon\int_{B_R^c}\frac{|\bar u|^{2_s^*(t)-1}|v|}{|x|^t}\bigg] \\
	 	&\leq \|K\|_{L^{\infty}(\Rn)}\bigg[\varepsilon \bigg(\int_{B_R^c}\frac{|v_n|^{2_s^*(t)}}{|x|^t}\;{\rm d}x\bigg)^{\tfrac{2_s^*(t)-1}{2_s^*(t)}}\bigg(\int_{B_R^c}
\frac{|v|^{2_s^*(t)}}{|x|^t}\bigg)^{\tfrac{1}{2_s^*(t)}}\\	&\qquad\qquad\qquad+C_\varepsilon\bigg(\int_{B_R^c}\frac{|\bar u|^{2_s^*(t)}}{|x|^t}\bigg)^{\tfrac{2_s^*(t)-1}{2_s^*(t)}}\bigg(\int_{B_R^c}\frac{|v|^{2_s^*(t)}}{|x|^t}\bigg)^{\tfrac{1}{2_s^*(t)}}\bigg]\\
	 	&\leq C\|K\|_{L^{\infty}(\Rn)}\bigg[\varepsilon \|v_n\|_{\gamma}^{2_s^*(t)-1}\bigg(\int_{B_R^c}
\frac{|v|^{2_s^*(t)}}{|x|^t}\bigg)^{\tfrac{1}{2_s^*(t)}}
+C_\varepsilon\|\bar u\|_{\gamma}^{2_s^*(t)-1}\bigg(\int_{B_R^c}
\frac{|v|^{2_s^*(t)}}{|x|^t}\bigg)^{\tfrac{1}{2_s^*(t)}}\bigg].	 	
\end{align*}
	Since $(\|v_n\|_{\ga})_n$ is uniformly bounded and $\frac{|v|^{2_s^*(t)}}{|x|^t}\in L^1(\Rn)$, given $\var>0$,  we can choose $R>0$ so large that	$$\bigg{|}\int_{B_R^c}K(x)\bigg\{\frac{|u_n|^{2_s^*(t)-2}u_n}{|x|^t}-\frac{|\bar u|^{2_s^*(t)-2}\bar u}{|x|^t}\bigg\}v\;{\rm d}x\bigg{|}<\var.$$
This completes the proof of claim 1.

Hence \eqref{24-5-1} yields that $\bar u$ is a solution of~\eqref{P}.
\medskip

\noindent\underline{\bf Step 3:} Here we show that $(u_n-\bar u)_n$ is a $(PS)$ sequence for $\IKto$ at the level $\ba-\IKtf (\bar u)$.
To see this, first we observe that as $n\to\infty$
$$\|u_n-\bar u\|_{\gamma}^2 = \|u_n\|_{\gamma}^2-\|\bar u\|_{\gamma}^2+o(1),$$
and by the Br\'ezis-Lieb lemma as $n\to\infty$
	$$\int_{\Rn} K(x)\frac{|u_n-\bar u|^{2_s^*(t)}}{|x|^t}{\rm d}x = \int_{\Rn}K(x)\frac{|u_n|^{2_s^*(t)}}{|x|^t}\;{\rm d}x - \int_{\Rn}K(x)\frac{|\bar u|^{2_s^*(t)}}{|x|^t}\;{\rm d}x+o(1).$$
Further as $u_n\rightharpoonup u$ and $f\in \dot{H}^s(\Rn)'$, we also have
$$
\prescript{}{\hms}{\langle}f, u_n{\rangle}_{\dot{H}^s}\To \prescript{}{\hms}{\langle}f,\bar u{\rangle}_{\dot{H}^s}.
$$
Therefore, as $n\to\infty$
\begin{align*}
	\IKto (u_n-\bar u) &= \frac{1}{2}\|u_n-\bar u\|_{\gamma}^2-\frac{1}{2_s^*(t)}\int_{\Rn}K(x)\frac{|u_n-\bar u|^{2_s^*(t)}}{|x|^t}\;{\rm d}x\\
	&= \frac{1}{2}\|u_n\|_{\gamma}^2-\frac{1}{2_s^*(t)}\int_{\Rn}K(x)\frac{|u_n|^{2_s^*(t)}}{|x|^t}\;{\rm d}x- \prescript{}{\hms}{\langle}f, u_n{\rangle}_{\dot H^s}\\
	&\,-\bigg\{\frac{1}{2}\|\bar u\|_{\gamma}^2-\frac{1}{2_s^*(t)}\int_{\Rn}K(x)\frac{|\bar u|^{2_s^*(t)}}{|x|^t}\;{\rm d}x - \prescript{}{\hms}{\langle}f, \bar u{\rangle}_{\dot H^s}\bigg\}+o(1)\\
	&= \IKtf(u_n)-\IKtf(\bar u)+o(1)\\
	&\To \ba-\IKtf(\bar u).
	\end{align*}
	Further, as $\prescript{}{\hms}{\big\langle}\dI(\bar u), v{\big\rangle}_{\dot H^s} =0$ for all $v\in\Hs$, we  obtain
\begin{align}\lab{24-5-3}
	 \prescript{}{\hms}{\big\langle}\dIo(u_n-\bar u), v{\big\rangle}_{\dot H^s}&=\langle u_n-\bar u,v\rangle_{\gamma}-\int_{\Rn}K(x)\frac{|u_n-\bar u|^{2_s^*(t)-2}(u_n-\bar u)v}{|x|^t}\;{\rm d}x\no\\
	 &=\langle u_n,v\rangle_{\gamma}-\int_{\Rn}K(x)\frac{|u_n|^{2_s^*(t)-2}u_nv}{|x|^t}\;{\rm d}x -\prescript{}{\hms}{\langle}f, v{\rangle}_{\dot H^s}\no\\
	 &-\bigg(\langle \bar u,v\rangle_{\gamma}-\int_{\Rn}K(x)\frac{|\bar u|^{2_s^*(t)-2}\bar u v}{|x|^t}\;{\rm d}x -\prescript{}{\hms}{\langle}f, v{\rangle}_{\dot H^s}\bigg)\no\\
	 &+\int_{\Rn}K(x)\bigg\{\frac{|u_n|^{2_s^*(t)-2}u_n}{|x|^t}-\frac{|\bar u|^{2_s^*(t)-2} \bar u}{|x|^t}\\
	 &\qquad\qquad\qquad-\frac{|u_n-\bar u|^{2_s^*(t)-2}(u_n-\bar u)}{|x|^t}\bigg\}v {\rm d}x\no\\
 &=o(1)+\int_{\Rn}K(x)\bigg\{\frac{|u_n|^{2_s^*(t)-2}u_n}{|x|^t}-\frac{|\bar u|^{2_s^*(t)-2} \bar u}{|x|^t}\no\\
	 &\qquad\qquad\qquad\qquad-\frac{|u_n-\bar u|^{2_s^*(t)-2}(u_n-\bar u)}{|x|^t}\bigg\}v {\rm d}x.\no
\end{align}
We observe that
\begin{align*}
 &\bigg|K\left\{|u_n|^{2_s^*(t)-2}u_k-|\bar{u}|^{2_s^*(t)-2}
 \bar{u}-|u_n-\bar{u}|^{2_s^*-2}(u_n-\bar{u})\right\}\bigg|\\
 &\hspace{4cm}\leq C\bigg(|u_n-\bar u|^{2^*_s(t)-2}|\bar u|+|u|^{2^*_s(t)-2}|u_n-\bar u|\bigg).
\end{align*}
Therefore, following the same method as in the proof of Claim~1 in Step~2,
we show  that as $n\to\infty$
\be\lab{25-5-8}\displaystyle\int_{\Rn}K(x)\bigg\{\frac{|u_n|^{2_s^*(t)-2}u_n}{|x|^t}-\frac{|\bar u|^{2_s^*(t)-2} \bar u}{|x|^t}-\frac{|u_n-\bar u|^{2_s^*(t)-2}(u_n-\bar u)}{|x|^t}\bigg\}v{\rm d}x=o(1)\ee
for all  $v\in\Hs$.
Plugging this back into \eqref{24-5-3}, we complete the proof of
Step~3.
\medskip
	
\noindent\underline{\bf Step 4:} Define $v_n:=u_n-\bar u.$
	Then $v_n\rightharpoonup 0$ in $\Hs$ and by Step~3, $(v_n)_n$ is a $(PS)$ sequence for $\IKto$ at the level $\ba-\IKtf(\bar u).$
Thus,
	\be\lab{25-5-4}
	\sup_{n\in\mathbb{N}}\|v_n\|_{\gamma}\leq C \quad\mbox{and}\quad
	\langle v_n, \va\rangle_{\ga}=\int_{\R{^N}}K(x)\frac{|v_n|^{2_s^*(t)-2}v_n\va}{|x|^t} \, {\rm d}x+o(1)\ee
as $n\to\infty$ for all $\va\in\Hs$. Therefore,	
$\|v_n\|_{\gamma}^2=\int_{\Rn}K(x)\frac{|v_n|^{2_s^*(t)}}{|x|^t}\;{\rm d}x+o(1)$. Thus, if $\int_{\Rn}K(x)\frac{|v_n|^{2_s^*(t)}}{|x|^t}\;{\rm d}x\longrightarrow 0$, then we are done when $k=l=0$
and the $(PS)$ sequence $(u_n)_n$ admits a strongly convergent subsequence.
	
If not, let $0<\delta <S_{\gamma,t,s}^{\frac{N-t}{2s-t}}\|K\|_{L^\infty(\mathbb R^N)}^{-\frac{N-2s}{2s-t}}$ such that
$$\limsup_{n\to\infty}\int_{\Rn}K(x)\frac{|v_n|^{2_s^*(t)}}{|x|^t}\;{\rm d}x>\delta.$$
Up to a subsequence, let $R_n>0$ be such that
$$\int_{B_{R_n}}K(x)\frac{|v_n|^{2_s^*(t)}}{|x|^t}\;{\rm d}x = \delta$$
and $R_n$ being minimal with this property. Define
$$w_n(x):= R_n^{\frac{N-2s}{2}}v_n(R_nx).$$ Therefore, $\|w_n\|_{\gamma}=\|v_n\|_{\gamma}$ and
	\be\label{PS3}
	\delta = \int_{B_{R_n}}K(x)\frac{|v_n|^{2_s^*(t)}}{|x|^t}\;{\rm d}x = \int_{B_1}K(R_nx)\frac{|w_n|^{2_s^*(t)}}{|x|^t}\;{\rm d}x.
	\ee
	Therefore, up to a subsequence
$$w_n \deb w\;\mbox{in }\Hs \quad\mbox{and}\quad
w_n\to w \mbox{ a.e. in }\Rn.$$
\medskip
Let us now distinguish two cases $w\neq 0$ and $w=0$.
\vspace{2mm}
	
\noindent\underline{\bf Step 5:} Assume that $w\neq 0.$

Since, $w_n\rightharpoonup w\neq 0$ and $v_n\rightharpoonup 0$, it follows that $R_n\to 0$ as $n\to\infty$. Next, we show that $w$ is a solution of \eqref{Q}.
Indeed, thanks to \eqref{25-5-4}, for any  $\phi\in C_c^{\infty}(\Rn)$
\begin{align}\lab{25-5-5}
	\langle w,\phi\rangle_{\gamma} &= \lim_{n\to\infty}\langle w_n,\phi\rangle_{\gamma}\no\\
	&=\lim_{n\to\infty} \frac{C_{N,s}}{2}\iint_{\R^{2N}}\frac{(w_n(x)-w_n(y))(\phi(x)-\phi(y))}{|x-y|^{N+2s}}\;{\rm d}x{\rm d}y-\gamma\int_{\Rn}\frac{w_n\phi}{|x|^{2s}}\;{\rm d}x\no\\
	&=\lim_{n\to\infty} \frac{C_{N,s}}{2}\iint_{\R^{2N}}\frac{R_n^{\frac{N-2s}{2}}(v_n(R_nx)-v_n(R_ny))(\phi(x)-\phi(y))}{|x-y|^{N+2s}}\;{\rm d}x{\rm d}y\no\\
&\qquad\qquad-\gamma\int_{\Rn}\frac{R_n^{\frac{N-2s}{2}}v_n(R_nx)\phi(x)}{|x|^{2s}}\;{\rm d}x\\
	&=\lim_{n\to\infty} \frac{C_{N,s}}{2}\iint_{\R^{2N}}\frac{R_n^{-\frac{N-2s}{2}}(v_n(x)-v_n(y))(\phi(\frac{x}{R_n})-\phi(\frac{y}{R_n}))}{|x-y|^{N+2s}}\;{\rm d}x{\rm d}y\no\\ &\qquad\qquad-\gamma\int_{\Rn}\frac{R_n^{-\frac{N-2s}{2}}v_n(x)\phi(\frac{x}{R_n})}{|x|^{2s}}\;{\rm d}x\no\\ &=\lim_{n\to\infty}\int_{\Rn}\!\!\!K(x)\frac{|v_n|^{2_s^*(t)-2}v_n}{|x|^t}R_n^{-\frac{N-2s}{2}}\phi\big(\frac{x}{R_n}\big)\;{\rm d}x=\lim_{n\to\infty}\int_{\Rn}K(R_nx)\!\!\!\frac{|w_n|^{2_s^*(t)-2}w_n}{|x|^t}\phi(x)\;{\rm d}x.\no
\end{align}
Clearly $K(R_nx)\frac{|w_n|^{2_s^*(t)-2}w_n}{|x|^t}\phi\to \frac{|w|^{2_s^*(t)-2}w}{|x|^t}\phi$ a.e. in $\Rn$,
since $K\in C(\Rn)$, with $K(0)=1$, and $w_n\to w$ a.e. in $\Rn$. Further, arguing as in the proof of Claim~1 in Step~2, we have
$K(R_nx)\frac{|w_n|^{2_s^*(t)-2}w_n}{|x|^t}\phi$ is uniformly integrable. Therefore,
as $\phi$ has compact support, using Vitali's convergence theorem we obtain
\be\lab{25-5-6}\lim_{n\to\infty}\int_{\Rn}K(R_nx)\frac{|w_n|^{2_s^*(t)-2}w_n}{|x|^t}\phi(x)\;{\rm d}x = \int_{\Rn}\frac{|w|^{2s^*(t)-2}w\phi}{|x|^t}{\rm d}x.\ee
	Combining \eqref{25-5-6} along with \eqref{25-5-5}, we conclude
that $w$ is a solution of \eqref{Q}.
		
	Define	 $$z_n(x):=v_n(x)-R_n^{-\frac{N-2s}{2}}w(\tfrac{x}{R_n}).$$
	
\noindent{\bf Claim 2:} $(z_n)_n$ is a $(PS)$ sequence for $\IKto$ at the level $\ba-\IKtf(\bar u)-\bar I_{1,0,0}^{\gamma}(w).$

To prove the claim, set $$\tilde{z}_n(x):=R_n^{\frac{N-2s}{2}}z_n(R_nx).$$ Then
$$ \tilde{z}_n(x)= w_n(x)-w(x) \quad\mbox{and}\quad  \|\tilde{z}_n\|_{\gamma}=\|w_n-w\|_{\gamma}=\|z_n\|_{\gamma}.$$
As $K(0)=1$ and $K$ is a continuous function, the  Br\'ezis-Lieb lemma and a  straight forward computation yield as $n\to\infty$
\begin{align*}
\int_{\Rn}K(R_nx)\frac{|w_n(x)|^{2_s^*(t)}}{|x|^t}\;{\rm d}x-\int_{\Rn}\frac{|w|^{2_s^*(t)}}{|x|^t}\;{\rm d}x
&=\int_{\Rn}\frac{\big|K^\frac{1}{2^*_s(t)}(R_nx)w_n-w\big|^{2_s^*(t)}}{|x|^t}\;{\rm d}x+o(1)\\
&=\int_{\Rn}K(R_nx)\frac{|w_n-w|^{2^*_s(t)}}{|x|^t}\;{\rm d}x+o(1).
\end{align*}
 Therefore, using the above relations, as $n\to\infty$
 \begin{align*}
\IKto(z_n)&=\frac{1}{2}\|z_n\|_{\gamma}^2-\frac{1}{2_s^*(t)}\int_{\Rn}K(x)\frac{|z_n|^{2_s^*(t)}}{|x|^t}\;{\rm d}x\\
&=\frac{1}{2}\|w_n-w\|_{\gamma}^2-\frac{1}{2_s^*(t)}\int_{\Rn}K(R_nx)\frac{|w_n-w|^{2_s^*(t)}}{|x|^t}\;{\rm d}x\\
&=\frac{1}{2}\big(\|w_n\|_{\gamma}^2-\|w\|_{\gamma}^2\big)-\frac{1}{2_s^*(t)}\int_{\Rn}K(R_nx)\frac{|w_n(x)|^{2_s^*(t)}}{|x|^t}\;{\rm d}x+\frac{1}{2_s^*(t)}\int_{\Rn}\frac{|w|^{2_s^*(t)}}{|x|^t}\;{\rm d}x+o(1)\\
&=\frac{1}{2}\|v_n\|_{\gamma}^2-\frac{1}{2_s^*(t)}\int_{\Rn}K(x)\frac{|v(x)|^{2_s^*(t)}}{|x|^t}\;{\rm d}x-\bigg(\frac{1}{2}\|w\|_{\gamma}^2-\frac{1}{2_s^*(t)}\int_{\Rn}\frac{|w|^{2_s^*(t)}}{|x|^t}\;{\rm d}x\bigg)+o(1)\\
&=\IKto(v_n)-\bar I_{1,0,0}^{\gamma}(w)+o(1)\\
&=\ba-\IKtf(\bar u)-\bar I_{1,0,0}^{\gamma}(w)+o(1).
\end{align*}
Next, let $\phi\in C^\infty_0(\Rn)$ be arbitrary and set $\phi_n(x):=R_n^{\frac{N-2s}{2}}\phi(R_nx)$. This in turn implies that
$\|\phi_n\|_{\gamma}=\|\phi\|_{\gamma}$ and $\phi_n\rightharpoonup 0$ in $\Hs$.  Therefore,
\begin{align}\lab{25-5-7}
\prescript{}{\hms}{\big\langle}\dIo(z_n), \phi{\big\rangle}_{\dot H^s} &=\langle z_n,\phi\rangle_{\gamma}-\int_{\Rn}K(x)\frac{|z_n|^{2_s^*(t)-2}z_n\phi}{|x|^t}\;{\rm d}x\no\\
	&= \langle \tilde{z}_n,\phi_n\rangle_{\gamma}-\int_{\Rn}K(R_nx)\frac{|\tilde{z}_n|^{2_s^*(t)-2}\tilde{z}_n\phi_n}{|x|^t}\;{\rm d}x\no\\
	&=\langle w_n-w,\phi_n\rangle_{\gamma}-\int_{\Rn}K(R_nx)\frac{|w_n-w|^{2_s^*(t)-2}(w_n-w)\phi_n}{|x|^t}\;{\rm d}x\no\\
	&= \langle w_n,\phi_n\rangle_{\gamma}-\int_{\Rn}K(R_nx)\frac{|w_n|^{2_s^*(t)-2}w_n\phi_n}{|x|^t}\;{\rm d}x\no\\
	&-\bigg(\langle w,\phi_n\rangle_{\gamma} -\int_{\Rn}\frac{|w|^{2_s^*(t)-2}w\phi_n}{|x|^t}\;{\rm d}x\bigg)\\
&+\int_{\Rn}(K(R_nx)-1)\frac{|w|^{2_s^*(t)-2}w\phi_n}{|x|^t}{\rm d}x\no\\
&+\int_{\Rn}\!\!\!K(R_nx)\bigg(\frac{|w_n|^{2_s^*(t)-2}w_n-|w|^{2_s^*(t)-2}w-|w_n-w|^{2_s^*(t)-2}(w_n-w)}{|x|^t}\bigg)\phi_n {\rm d}x\no\\
	&=\langle v_n,\phi\rangle_{\gamma}-\int_{\Rn}K(x)\frac{|v_n|^{2_s^*(t)-2}v_n\phi}{|x|^t}\;{\rm d}x-\prescript{}{\hms}{\big\langle}\bar I_{1,0,0}^{\gamma}{'}(w), \phi_n{\big\rangle}_{\dot H^s}+I^1_n+I^2_n\no\\
	&= \prescript{}{\hms}{\big\langle}\dIo(v_n), \phi{\big\rangle}_{\dot H^s}-0+I^1_n+I^2_n=o(1)+I^1_n+I^2_n.\no
\end{align}
	Now
$$I_n^1:=\int_{B_R}\big(K(R_nx)-1\big)\frac{|w|^{2_s^*(t)-2}w\phi_n}{|x|^t}\;{\rm d}x
	+\int_{B_R^c} \big(K(R_nx)-1\big)\frac{|w|^{2_s^*(t)-2}w\phi_n}{|x|^t}\;{\rm d}x.$$
Note that as $\frac{|w|^{2_s^*(t)}}{|x|^t}\in L^1(\Rn)$,	for $\varepsilon>0$ there exists $R=R(\varepsilon)>0$ such that
\begin{align*} \bigg{|}\int_{B_R^c}\big(K(R_nx)-1\big)\frac{|w|^{2_s^*(t)-2}w\phi_n}{|x|^t}\;{\rm d}x\bigg{|}&\leq C\bigg(\int_{B_R^c}\frac{|w|^{2^*_s(t)}}{|x|^t}{\rm d}x\bigg)^{\frac{2_s^*(t)-1}{2_s^*(t)}}\bigg(\int_{\Rn}\frac{|\phi_n|^{2_s^*(t)}}{|x|^t}\;{\rm d}x\bigg)^{\frac{1}{2_s^*(t)}}\\
	&\leq C\bigg(\int_{B_R^c}\frac{|w|^{2^*_s(t)}}{|x|^t}{\rm d}x\bigg)^{\frac{2_s^*(t)-1}{2_s^*(t)}}\|\phi\|_{\ga}<\varepsilon.
\end{align*}
On the other hand, as $K\in C(\Rn)$ and $\lim_{|x|\to\infty}K(x)=1$ implies that $K\in L^\infty(\Rn)$, applying the H\"{o}lder inequality followed by the Hardy-Sobolev inequality, it is easy to see that  $$\big(K(R_nx)-1\big)\frac{|w|^{2_s^*(t)-2}w\phi_n}{|x|^t}$$ is uniformly integrable. Therefore, using Vitalis convergence theorem, we get	$$\int_{B_R}\big(K(R_nx)-1\big)\frac{|w|^{2_s^*(t)-2}w\phi_n}{|x|^t}\;{\rm d}x=o(1).$$
Hence, $I^1_n=o(1)$ as $n\to\infty$.	
	
Next, we aim to show that  $$I^2_n:=\int_{\Rn}K(R_nx)\bigg\{\frac{|w_n|^{2_s^*(t)-2}w_n-|w|^{2_s^*(t)-2}w-|w_n-w|^{2_s^*(t)-2}(w_n-w)}{|x|^t}\bigg\}\phi_n {\rm d}x=o(1).$$
Indeed, this follows as in the proof of \eqref{25-5-8}, since $\int_{\Rn}\frac{|\phi_n|^{2_s^*(t)}}{|x|^t}\;{\rm d}x=\int_{\Rn}\frac{|\phi|^{2_s^*(t)}}{|x|^t}\;{\rm d}x<\infty$.
Hence, from \eqref{25-5-7} we conclude the proof of Claim 2.	
\medskip

\noindent\underline{\bf Step 6:}  Assume that $w= 0$.
\vspace{2mm}

	Let $\va\in C^\infty_0\big(B_1\big)$, with $0\leq\va\leq 1$. Set $\psi_n(x) :=[\va(\tfrac{x}{R_n})]^2v_n(x)$. Clearly $(\psi_n)_n$
is a bounded sequence in~$\Hs$. Thus,
\begin{align*}
o(1)&=\prescript{}{\hms}{\big\langle}\dIo(v_n), \psi_n{\big\rangle}_{\dot H^s} \\
&=\langle v_n,\psi_n\rangle_{\gamma}-\int_{\Rn}K(x)\frac{|v_n|^{2_s^*(t)-2}v_n\psi_n}{|x|^t}{\rm d}x\\	&=\frac{C_{N,s}}{2}\iint_{\R^{2N}}\frac{(v_n(x)-v_n(y)\big(\va^2(\frac{x}{R_n})v_n(x)-\va^2(\frac{y}{R_n})v_n(y)\big)}{|x-y|^{N+2s}}{\rm d}x{\rm d}y-\gamma \int_{\Rn}\frac{v_n^2(x)\va^2(\frac{x}{R_n})}{|x|^{2s}}{\rm d}x\\
&\;\;\quad-\int_{\Rn}K(x)\frac{\va^2(\frac{x}{R_n})|v_n|^{2_s^*(t)}}{|x|^t}{\rm d}x\\ &=\frac{C_{N,s}}{2}\iint_{\R^{2N}}\frac{\big(v_n(R_nx)-v_n(R_ny)\big)\big(\va^2(x)v_n(R_nx)-\va^2(y)v_n(R_ny)\big)R_n^{N-2s}}{|x-y|^{N+2s}}{\rm d}x{\rm d}y\\
&\qquad-\gamma \int_{\Rn}\frac{v_n^2(R_nx)\va^2(x)R_n^{N-2s}}{|x|^{2s}}{\rm d}x
-\int_{\Rn}K(x)\frac{|v_n|^{2_s^*(t)-2}\big(\va(\frac{x}{R_n})v_n\big)^{2}}{|x|^t}{\rm d}x.
\end{align*}
Therefore
\begin{equation}\label{PS1}\begin{aligned}
&\frac{C_{N,s}}{2}\iint_{\R^{2N}}\!\!\!\frac{\big(v_n(R_nx)-v_n(R_ny)\big)\big(\va^2(x)v_n(R_nx)-\va^2(y)v_n(R_ny)\big)R_n^{N-2s}}{|x-y|^{N+2s}}{\rm d}x{\rm d}y\\
&\quad\qquad-\gamma \int_{\Rn}\!\!\!\frac{v_n^2(R_nx)\va^2(x)R_n^{N-2s}}{|x|^{2s}}{\rm d}x=\int_{\Rn}\!\!\!
K(x)\frac{|v_n|^{2_s^*(t)-2}\big(\va(\frac{x}{R_n})v_n\big)^{2}}{|x|^t}{\rm d}x+o(1).
\end{aligned}\end{equation}
Now,
\begin{align}\label{26-5-1}
\mbox{RHS of \eqref{PS1}}
&=\int_{B_1}K(R_nx)\frac{|v_n(R_nx)|^{2_s^*(t)-2}\big(\va(x)v_n(R_nx)\big)^2R_n^{N-t}}{|x|^t}\;{\rm d}x+o(1)\no\\
&=\int_{B_1}\frac{\big|K(R_nx)^\frac{1}{2_s^*(t)-2}w_n(x)\big|^{2_s^*(t)-2}\big(\va(x)w_n(x)\big)^2}{|x|^t}\;{\rm d}x+o(1)\no\\
&\leq\bigg(\int_{B_1}K(R_nx)^\frac{2_s^*(t)}{2_s^*(t)-2}\frac{|w_n(x)|^{2_s^*(t)}}{|x|^t}\;{\rm d}x\bigg)^{\frac{2_s^*(t)-2}{2_s^*(t)}}\cdot\no\\
&\qquad\times\bigg(\int_{\Rn}\frac{|\va w_n|^{2_s^*(t)}}{|x|^t}\;{\rm d}x\bigg)^{\frac{2}{2_s^*(t)}}+o(1)\\
	&\leq\frac{\|K\|_{L^\infty}^\frac{2}{2_s^*(t)}}{S_{\gamma,t,s}}
	 \bigg(\int_{B_1}K(R_nx)\frac{|w_n|^{2_s^*(t)}}{|x|^t}\;{\rm d}x\bigg)^{\frac{2_s^*(t)-2}{2_s^*(t)}}\|\va w_n\|_{\gamma}^2+o(1)\no\\
&\leq\frac{\|K\|_{L^\infty}^\frac{2}{2_s^*(t)}\delta^{\frac{2s-t}{N-t}}}{S_{\gamma,t,s}}\|\va w_n\|_{\gamma}^2 +o(1)\no\\
	&<\|\va w_n\|_{\gamma}^2 +o(1)\;\;\mbox{(By the choice of $\delta$ fixed
in Step~4)}.\no
\end{align}

\noindent{\bf Claim 3:} As $n\to\infty$
\be
\mbox{LHS of \eqref{PS1}}= \|\va w_n\|_{\gamma}^2+o(1)\lab{26-5-6}.\ee
Indeed,
\begin{align}\label{PS2}
\mbox{LHS of \eqref{PS1}}&=  \frac{C_{N,s}}{2}\iint_{\R^{2N}}\frac{\big(v_n(R_nx)-v_n(R_ny)\big)\big(\va^2(x)v_n(R_nx)-\va^2(y)v_n(R_ny)\big)R_n^{N-2s}}{|x-y|^{N+2s}}\;{\rm d}x{\rm d}y\no\\
&\qquad\qquad-\gamma \int_{\Rn}\frac{|\va w_n|^2}{|x|^{2s}}{\rm d}x\no\\	&=\frac{C_{N,s}}{2}\iint_{\R^{2N}}\frac{\big(w_n(x)-w_n(y)\big)\big(\va^2(x)w_n(x)-\va^2(y)w_n(y)\big)}{|x-y|^{N+2s}}\;{\rm d}x{\rm d}y\no\\
&\qquad\qquad-\gamma \int_{\Rn}\frac{|\va w_n|^2}{|x|^{2s}}{\rm d}x\\	
&=\frac{C_{N,s}}{2} \iint_{\R^{2N}}\frac{|\va(x)w_n(x)-\va(y)w_n(y)|^2}{|x-y|^{N+2s}}{\rm d}x{\rm d}y-\gamma\int_{\Rn}\frac{|\va w_n|^2}{|x|^{2s}}\;{\rm d}x\no\\
	&\qquad\qquad-\frac{C_{N,s}}{2} \iint_{\R^{2N}}\frac{(\va(x)-\va(y))^{2}w_n(x)w_n(y)}{|x-y|^{N+2s}}{\rm d}x{\rm d}y\no\\
&=\|\va w_n\|_{\ga}^2-\frac{C_{N,s}}{2} \iint_{\R^{2N}}\frac{(\va(x)-\va(y))^{2}w_n(x)w_n(y)}{|x-y|^{N+2s}}{\rm d}x{\rm d}y.\no
\end{align}
	Now,
\begin{align*}
	\iint_{\R^{2N}}\frac{(\va(x)-\va(y))^{2}w_n(x)w_n(y)}{|x-y|^{N+2s}}{\rm d}x{\rm d}y &= \int_{x\in B_1}\int_{y\in B_1}\quad +\int_{x\in B_1}\int_{y\in B_1^c}\quad +\int_{x\in B_1^c}\int_{y\in B_1}\\
	&=:\mathcal{I}_n^1 +\mathcal{I}_n^2 +\mathcal{I}_n^3.\no
\end{align*}
Of course, $\mathcal{I}_n^2=\mathcal{I}_n^3$, as the integral is symmetric with respect to $x$ and $y$.
\begin{align}\lab{26-5-3}
\mathcal{I}_n^1&= \int_{x\in B_1}\int_{y\in B_1}\frac{(\va(x)-\va(y))^{2}w_n(x)w_n(y)}{|x-y|^{N+2s}}\;{\rm d}x{\rm d}y\no\\
	&\leq C\int_{x\in B_1}\int_{y\in B_1}\frac{|w_n(x)||w_n(y)|}{|x-y|^{N+2s-2}}\;{\rm d}x{\rm d}y\no\\
	&\leq C\bigg(\int_{x\in B_1}\int_{y\in B_1}\frac{|w_n(x)|^2}{|x-y|^{N+2s-2}}\;{\rm d}x{\rm d}y\bigg)^{\tfrac{1}{2}}\cdot\no\\
&\qquad\quad\times\bigg(\int_{x\in B_1}\int_{y\in B_1}\frac{|w_n(y)|^2}{|x-y|^{N+2s-2}}{\rm d}x{\rm d}y\bigg)^{\tfrac{1}{2}}\\
	&\leq C\int_{x\in B_1}\int_{y\in B_1}\frac{|w_n(x)|^2}{|x-y|^{N+2s-2}}\;{\rm d}x{\rm d}y\no\\
	&\leq C\int_{x\in B_1}\bigg(\int_{|z|<2}\frac{1}{|z|^{N+2s-2}}\;{\rm d}z\bigg)|w_n(x)|^2{\rm d}x\no\\
	&\leq C\|w_n\|_{L^2(B_1)}^2
= o(1)\;\;\mbox{\big(as $w=0$ implies $w_n\to 0$ in $L^2_{\rm loc}(\Rn)$\big).}\no
\end{align}
Furthermore,
\begin{align}\lab{26-5-4}
\mathcal{I}_n^2&=\int_{x\in B_1}\int_{y\in B_1^c}\frac{(\va(x)-\va(y))^{2}w_n(x)w_n(y)}{|x-y|^{N+2s}}{\rm d}x{\rm d}y\no\\
	&\leq  \int_{x\in B_1}\int_{y\in B_1^c\cap \{|x-y|\leq 1\}}+ \int_{x\in B_1}\int_{y\in B_1^c\cap \{|x-y|\geq 1\}}\\
	&=: \mathcal{I}_n^{21}+\mathcal{I}_n^{22},\no
\end{align}
where
\begin{align*}	
	\mathcal{I}_n^{21}&\leq C\bigg(\int_{x\in B_1}\int_{y\in B_1^c\cap \{|x-y|\leq 1\}}\frac{|w_n(x)|^2}{|x-y|^{N+2s-2}}\;{\rm d}y{\rm d}x\bigg)^{\tfrac{1}{2}}\cdot\\
	&\qquad\times\bigg(  \int_{x\in B_1}\int_{y\in B_1^c\cap \{|x-y|\leq1\}}\frac{|w_n(y)|^2}{|x-y|^{N+2s-2}}\;{\rm d}y{\rm d}x\bigg)^{\tfrac{1}{2}}\\
	&=:CJ_n^{1}\cdot J_n^{2}.\\
\end{align*}
Now,
$$
|J_n^{1}|^2\leq\int_{x\in B_1}\bigg(\int_{|z|<1}\frac{1}{|z|^{N+2s-2}}\;{\rm d}z\bigg)|w_n(x)|^2{\rm d}x \leq C\|w_n\|_{L^2(B_1)}^2	= o(1),$$
and
\begin{align*}
	|J_n^2|^2&= \int_{x\in B_1}\int_{y\in B_1^c}\frac{\mathbf{1}_{\{|x-y|<1\}}(x,y)|w_n(y)|^2}{|x-y|^{N+2s-2}}\;{\rm d}y{\rm d}x\\
	&\leq\int_{y\in B_1^c}\bigg(\int_{x\in B_1}\frac{\mathbf{1}_{\{|x-y|<1\}}(x,y)}{|x-y|^{N+2s-2}}\;{\rm d}x\bigg)|w_n(y)|^2\;{\rm d}y\\
	&\leq \int_{y\in B_2}\bigg(\int_{x\in B_1}\frac{\mathbf{1}_{\{|x-y|<1\}}(x,y)}{|x-y|^{N+2s-2}}\;{\rm d}x\bigg)|w_n(y)|^2\;{\rm d}y\\
	&\leq C\|w_n\|_{L^2(B_2)}^2\leq C'.
\end{align*}
Therefore,  $\mathcal{I}_n^{21}= o(1)$ as $n\to\infty$.
Moreover,
\begin{align*}
\mathcal{I}_n^{22}&=\int_{x\in B_1}\int_{y\in B_1^c\cap \{|x-y|\geq 1\}} \frac{|w_n(x)||w_n(y)||\va (x)-\va (y)|^2}{|x-y|^{N+2s}}\;{\rm d}y{\rm d}x\\
	&\leq C\int_{x\in B_1}\int_{y\in B_1^c\cap \{|x-y|\geq 1\}} \frac{|w_n(x)||w_n(y)|}{|x-y|^{N+2s}}\;{\rm d}y{\rm d}x\\
	&\leq C\bigg(\int_{x\in B_1}\int_{y\in B_1^c\cap \{|x-y|\geq 1\}} \frac{|w_n(x)|^2}{|x-y|^{N+2s}}\;{\rm d}y{\rm d}x\bigg)^{\frac{1}{2}}\cdot\\
	&\qquad\qquad\times\bigg(\int_{x\in B_1}\int_{y\in B_1^c\cap \{|x-y|\geq 1\}} \frac{|w_n(y)|^2}{|x-y|^{N+2s}}\;{\rm d}y{\rm d}x\bigg)^{\frac{1}{2}}\\
	&\leq C\bigg(\int_{x\in B_1}\bigg(\int_{|z|\geq 1}\frac{1}{|z|^{N+2s}}\;{\rm d}z\bigg)|w_n(x)|^2{\rm d}x\bigg)^{\frac{1}{2}}\cdot\\
	&\qquad\qquad\times\bigg(\int_{x\in B_1}\int_{|z|\geq 1} \frac{|w_n(x+z)|^2}{|x+z|^{2s}}\frac{|x+z|^{2s}}{|z|^{N+2s}}\;{\rm d}z{\rm d}x\bigg)^{\frac{1}{2}}\\
	&\leq C' \|w_n\|_{L^2(B_1)}\bigg[\int_{x\in B_1}\bigg(\int_{\Rn} \frac{|w_n(x+z)|^2}{|x+z|^{2s}}{\rm d}z\bigg){\rm d}x\bigg]^{\frac{1}{2}},
	\end{align*}
since $|z|\geq 1$ and $|x|<1$ implies $\frac{|x+z|^{2s}}{|z|^{N+2s}}\leq  C$.  Therefore, using the Hardy inequality, we obtain from the last of the above estimate that as $n\to\infty$
$$
\mathcal{I}_n^{22} \leq C'' \|w_n\|_{L^2(B_1)}\|w_n\|_{\Hs}^2 = o(1).
$$
Putting the above estimates together, we obtain from \eqref{26-5-4} that $\mathcal{I}_n^2=o(1)$ as $n\to\infty$. This, along with~\eqref{26-5-3}, concludes the proof of Claim~3.
\vspace{2mm}
	
Combining Claim 3 with \eqref{26-5-1} yields
\be\lab{26-5-9}
\|\va w_n\|_{\gamma}= o(1) \mbox{ as }n\to\infty.
\ee
Substituting this into \eqref{26-5-6} and comparing with \eqref{PS1} yields
as $n\to\infty$
 $$\int_{\Rn}K(R_nx)\frac{\va^2(x)|w_n(x)|^{2_s^*(t)}}{|x|^t}\;{\rm d}x =o(1).$$
Therefore, \be\lab{26-5-8}\int_{B_r}K(R_nx)\frac{|w_n|^{2_s^*(t)}}{|x|^t}\;{\rm d}x =o(1), \quad\mbox{for any}\quad 0<r<1.\ee
 But this contradicts  \eqref{PS3} when $t>0$. Therefore, $w=0$
cannot happen in the case $t>0$, i.e.,
 $$t>0 \implies w\neq 0.$$

Consequently, from now onwards, we restrict ourselves to the case $t=0$ and $w=0$.
 \medskip

\noindent\underline{\bf Step 7}: 	
Let $t=0$ and $w=0$. First we consider the tight case, $(v_n)_n\subseteq \dot H^s_0(B_R)$, for some fixed ball of radius $R>0$ (where $\dot H^s_0(B_R)$ is the closure of $C_0^{\infty}(B_R)$ with respect to the $\Hs$ norm). The remaining case will be obtained by a splitting argument together with a Kelvin transform.

	\medskip

 Therefore, in view of \eqref{PS3} and  \eqref{26-5-8}, using the concentration-compactness principle in the tight case \cite{Li}, it follows that in the sense of measure,
	\be\label{PS4}
	K(R_nx)|w_n|^{2_s^*}dx\big{|}_{\{|x|\leq 1\}}\overset{*}{\rightharpoonup} \sum_{j}C_{x_j}\delta_{x_j},
	\ee
where $x_j\in\Rn$ satisfies $|x_j|=1.$ Let $\bar C:=\max_{j}C_{x_j}$ and define
	\be\label{PS5}
	Q_n(r):= \sup_{y\in\Rn}\int_{B_r(y)}K(R_nx)|w_n|^{2_s^*}\;{\rm d}x.
	\ee
	Clearly, $Q_n(r)>{C}/{2}$ for each $r>0$ large enough. Moreover, \eqref{PS4} gives
$$\liminf_{n\to\infty}Q_n(r)\ge\frac{C}{2}.$$
Hence, there exist sequences $(s_n)_n\subset \R^+$ and $(q_n)_n\subset\Rn$ such that $s_n\to 0$ and $|q_n|>{1}/2$ and
	\be\label{PS6}
	\frac{C}{2}=\sup_{q\in\Rn}\int_{B_{s_n}(q)}K(R_nx)w_n^{2_s^*}{\rm d}x=\int_{B_{s_n}(q_n)}K(R_nx)w_n^{2_s^*}\;{\rm d}x.
	\ee
	Define $\theta_n(x):= s_n^{\tfrac{N-2s}{2}}w_n(s_nx+q_n)$. Thus $\|\theta_n\|_{\ga}=\|w_n\|_{\ga}$ for any  $n\in\mathbb{N}$. Consequently, up to a subsequence, there exists $\theta\in\Hs$ such that $\theta_n\rightharpoonup \theta$ in $\Hs$ and $\theta_n\to \theta$ a.e. in $\Rn$.
	
	First note that $\theta\neq 0$. Otherwise, choosing $\va\in C^\infty_0\big(B_1(x)\big)$, with $0\leq \va\leq 1$, for an arbitrary but fixed $x\in\Rn$, and proceeding exactly as in obtaining \eqref{26-5-9}, we are able to show that $\theta_n\to 0$ in $L^{2^*_s}_{\rm loc}(\Rn)$. On the other hand, from \eqref{PS6} it follows that
	$$\int_{B_1}K(s_nR_nx+q_n)\theta_n^{2_s^*}\;{\rm d}x =\frac{C}{2}>0.$$
which leads to a contradiction. Thus, $\theta\neq 0.$
Recall  that
$$\theta_n(x)= s_n^{\frac{N-2s}{2}}w_n(s_nx+q_n)=(s_nR_n)^{(\frac{N-2s}{2})}v_n(s_nR_nx+R_nq_n).
$$
Define $r_n=s_nR_n=o(1)$ and $y_n=R_nq_n$. Hence, $\frac{r_n}{|y_n|}<2s_n=o(1)$ and, up to a subsequence, $y_n\to y$ in~$\Rn$. From Lemma \ref{L1}, we deduce that
	$$\theta=K(y)^{-\frac{N-2s}{4s}}W^{\tau,a}\quad \mbox{for some } \tau>0,\;a\in\Rn, $$
	where $W$ is a solution of \eqref{W}
	and that $n\mapsto\tilde{v}_n(x):= v_n(x)-K(y)^{\frac{4s}{N-2s}}W^{r_n\tau,y_n+r_na}(x)$
	is a $(PS)$ sequence for $\IKto$ at level $\ba-\IKtf(\bar u)-K(y)^{-\frac{N-2s}{2s}}\bar I_{1,0,0}^{(0)}(W),$ where $W$ is a solution of \eqref{W}.
\medskip
	
	In summary, in both cases $t>0$ and $t=0$, starting from a $(PS)$ sequence $(v_n)_n$ of $\IKto$ we have found another $(PS)$ sequence $(\tilde{v}_n)_n$ of $\IKto$ at a strictly lower level, with a fixed minimum amount of decrease. Since $\sup_n \|v_n\|_{\gamma}\leq C<\infty$, the process should stop after finitely many steps.
\medskip
	
\noindent\underline{\bf Step 8:} When $t=0$ we only dealt with the case $(v_n)_n\subset \dot H^s_0(B_R)$ for some fixed $R>0.$ Now we are going to relax the assumption $(v_n)_n\subset \dot H^s_0(B_R)$.
	
	Let us define
$$
\tilde{f}(k):=\liminf_{n\to\infty}\int_{B_{k+1}\setminus B_{k}}K(x)|v_n|^{2_s^*}\;{\rm d}x.
$$
We claim that $\tilde{f}(k)=0$ for all but finitely many $k$'s.

Indeed, if $\tilde{f}(k)>0$ for some $k$, then $\liminf_{n\to\infty}\int_{B_{k+1}\setminus B_{k}}K(x)|v_n|^{2_s^*}\;{\rm d}x >0.$ Therefore,
\be\lab{29-5-1}
\liminf_{n\to\infty}\int_{B_{k+1}\setminus B_{k}}|v_n|^{2_s^*}\;{\rm d}x >0.
\ee
By Step~6, for any $\va\in C^\infty_0(\Rn)$ as $n\to\infty$
\begin{equation}\lab{29-5-2}\hspace{-.3cm}\begin{aligned}
	\|K\|_{L^\infty}^\frac{2}{2^*_s}\bigg(\!\int_{\mbox{supp} (\va)}\!\!\!K(R_nx)|w_n|^{2_s^*}{\rm d}x\bigg)^{\frac{2_s^*-2}{2_s^*}}\!\bigg(\!\int_{\Rn}\!\!\!|\va w_n|^{2_s^*}{\rm d}x\bigg)^{\frac{2}{2_s^*}}
&\geq \|\va w_n\|_{\gamma}^2+o(1)\\
&\geq S_{\gamma,0,s}\bigg(\!\int_{\R^N}\!\!\!|\va w_n|^{2_s^*}\bigg)^{\tfrac{2}{2_s^*}}\!+o(1).
\end{aligned}\end{equation}
Fix any $\eps>0$ and choose $\va\in C^\infty_0(\Rn)$
	such that $\va\equiv 1$ in $B_{k+1}\setminus B_{k}$ and supp$(\va)\subseteq B_{k+1+\eps}\setminus B_{k-\eps}$ and $0\leq\va\leq 1$. Define, $\va_n(x)=\va(R_nx)$. Then	$$\liminf_{n\to\infty}\int_{\Rn}|\va_nw_n|^{2^*_s}dx=\liminf_{n\to\infty}\int_{\Rn}|\va v_n|^{2^*_s}dx>\liminf_{n\to\infty}\int_{B_{k+1}\setminus B_{k}}|v_n|^{2^*_s}dx>0.$$
Now \eqref{29-5-2}, with $\va=\va_n$, yields as $n\to\infty$	
$$\int_{B_{k+1+\eps}\setminus B_{k-\eps}}K(x)|v_n|^{2_s^*}\;{\rm d}x\geq \|K\|_{L^\infty}^{-\frac{N-2s}{2s}}S_{\gamma,0,s}^\frac{N}{2s}+o(1).$$
Combining the above, as $\eps>0$ is arbitrary, we obtain $\tilde f(k)\geq \|K\|_{L^\infty}^{-\frac{N-2s}{2s}}S_{\gamma,0,s}^\frac{N}{2s}$. Therefore, since $(v_n)_n$ is bounded in $L^{2^*_s}(\Rn)$, it follows that $\tilde{f}(k)=0$ for all but finitely many $k$'s and this completes the proof of the claim.
\medskip
	
Now given such a $k$ for which $\tilde f(k)=0$, we take a cut-off function $\chi\in C^\infty_0(\Rn)$ such that $\chi\equiv 1$ on $B_{k}$ and $\chi\equiv 0$ on $B_{k+1}^c$ and $0\leq \chi\leq 1$. We shall show that both $(\chi v_n)_n$ and $\big((1-\chi)v_n\big)_n$ are $(PS)$ sequences for $I^{\ga}_{K,0,0}$. Indeed for $h\in C^\infty_0(\Rn)$ as $n\to\infty$
\begin{align}\label{PS8}
\langle \chi v_n,h\rangle_{\ga} &= \frac{C_{N,s}}{2}\iint_{\R^{2N}}\frac{\big(\chi (x)v_n(x)-\chi (y)v_n(y)\big)\big(h(x)-h(y)\big)}{|x-y|^{N+2s}}\;{\rm d}x{\rm d}y-\ga\int_{\Rn}\frac{\chi v_n h}{|x|^{2s}}{\rm d}x\no\\
    &= \frac{C_{N,s}}{2}\iint_{\R^{2N}}\frac{\big(v_n(x)-v_n(y)\big)\big(\chi(x)h(x)-\chi(y)h(y)\big)}{|x-y|^{N+2s}}\;{\rm d}x{\rm d}y-\ga\int_{\Rn}\frac{v_n (\chi h)}{|x|^{2s}}{\rm d}x\no\\
    &\quad +\frac{C_{N,s}}{2}\iint_{\R^{2N}}\frac{\big(\chi (x)-\chi (y)\big)h(x)v_n(y)}{|x-y|^{N+2s}}\;{\rm d}x{\rm d}y -\frac{C_{N,s}}{2}\iint_{\R^{2N}}\frac{\big(\chi (x)-\chi (y)\big)h(y)v_n(x)}{|x-y|^{N+2s}}\;{\rm d}x{\rm d}y \\
    &= \langle v_n,\chi h\rangle_{\ga}+C_{N,s}\iint_{\R^{2N}}\frac{\big(\chi (x)-\chi (y)\big)h(x)v_n(y)}{|x-y|^{N+2s}}\;{\rm d}x{\rm d}y\no\\
    &= \int_{\Rn}K(x)|v_n|^{2_s^*-2}v_n(\chi h)\;{\rm d}x+C_{N,s}\mathbb{I}_n+o(\|h\|),\no
\end{align}
where $\mathbb{I}_n:= \displaystyle\iint_{\R^{2N}}\frac{\big(\chi (x)-\chi (y)\big)h(x)v_n(y)}{|x-y|^{N+2s}}\;{\rm d}x{\rm d}y.$
\vspace{2mm}

\noindent{\bf Claim 4:} $\mathbb{I}_n=o(\|h\|_{\ga})$ as $n\to\infty$.
	
Indeed,
	$$\mathbb{I}_n\leq \bigg(\iint_{\R^{2N}}\frac{|\chi (x)-\chi (y)|^2h^2(x)}{|x-y|^{N+2s}}\;{\rm d}x{\rm d}y\bigg)^\frac{1}{2}\bigg(\iint_{\R^{2N}}\frac{|\chi (x)-\chi (y)|^2v_n^2(y)}{|x-y|^{N+2s}}\;{\rm d}x{\rm d}y\bigg)^\frac{1}{2}.$$
Now,
 \begin{align*}
\iint_{\R^{2N}}\frac{|\chi (x)-\chi (y)|^2v_n^2(y)}{|x-y|^{N+2s}}\;{\rm d}x{\rm d}y&=\int_{y\in B_{k+1}}\int_{x\in B_{k+1}} +\int_{y\in B_{k+1}}\int_{x\in B_{k+1}^c}+
\int_{y\in B_{k+1}^c}\int_{x\in B_{k+1}}\\
&=:\mathbb{I}_n^1+	\mathbb{I}_n^2+\mathbb{I}_n^3.
\end{align*}
Since $v\deb 0$ in $\Hs$ implies $v_n\to 0$ in $L^2_{\rm loc}(\Rn)$, we see that as $n\to\infty$
 \begin{align}\lab{31-5-1}
\mathbb{I}_n^1&=\int_{y\in B_{k+1}}\int_{x\in B_{k+1}} \frac{|\chi (x)-\chi (y)|^2v_n^2(y)}{|x-y|^{N+2s}}{\rm d}x{\rm d}y\no\\
&\leq C\!\int_{y\in B_{k+1}}\!\!\int_{x\in B_{k+1}}\!\!\frac{v_n^2(y)}{|x-y|^{N+2s-2}}\;{\rm d}x{\rm d}y\no\\
&\leq C\int_{y\in B_{k+1}}\!\!\!\bigg(\int_{x\in B_{k+1}\cap \{|x-y|<1\}}\!\frac{{\rm d}x}{|x-y|^{N+2s-2}}+\int_{x\in B_{k+1}\cap \{|x-y|\geq1\}}\!\!\!{\rm d}x\bigg)v_n^2(y){\rm d}y\\
&\leq C'\int_{y\in B_{k+1}}v_n^2(y){\rm d}y\no\\
&=o(1);\no
\end{align}
\begin{align}\lab{31-5-2}
\mathbb{I}_n^2&=\int_{y\in B_{k+1}}\int_{x\in B_{k+1}^c} \frac{|\chi (x)-\chi (y)|^2v_n^2(y)}{|x-y|^{N+2s}}\;{\rm d}x{\rm d}y\no\\
&\leq C\!\int_{y\in B_{k+1}}\!\!\!\bigg(\int_{x\in B_{k+1}^c\cap\{|x-y|\leq 1\}}\!\frac{{\rm d}x}{|x-y|^{N+2s-2}}\!+\!\int_{x\in B_{k+1}^c\cap\{|x-y|\geq 1\}}\!\frac{{\rm d}x}{|x-y|^{N+2s}} \bigg)v_n^2(y){\rm d}y\\
&\leq C''\int_{y\in B_{k+1}}v_n^2(y){\rm d}y
=o(1);\no
\end{align}
\begin{align*}\mathbb{I}_n^3&=\int_{y\in B_{k+1}^c}\int_{x\in B_{k+1}}\frac{|\chi (x)-\chi (y)|^2v_n^2(y)}{|x-y|^{N+2s}}\;{\rm d}x{\rm d}y\\
&=\int_{x\in B_{k+1}}\int_{y\in B_{k+1}^c}\frac{|\chi (x)-\chi (y)|^2v_n^2(y)}{|x-y|^{N+2s}}\;{\rm d}y{\rm d}x\\
&=\int_{x\in B_{k+1}}\int_{y\in B_{k+1}^c\cap\{|x-y|\geq 1\}}  +\int_{x\in B_{k+1}}\int_{y\in B_{k+1}^c\cap\{|x-y|\leq 1\}}\\
&=: \mathbb{I}_n^{31}+\mathbb{I}_n^{32}.
\end{align*}
For estimating $\mathbb{I}_n^{31}$,  we choose $\eps>0$ arbitrary and $R>>k+1$ so that
 \begin{align}\lab{31-5-3}
  \mathbb{I}_n^{31}&=\int_{x\in B_{k+1}}\int_{y\in B_{k+1}^c\cap\{|x-y|\geq 1\}}
  \frac{\chi^2(x)v_n^2(y)}{|x-y|^{N+2s}}\;{\rm d}y{\rm d}x\no\\
  &\leq\int_{x\in B_{k+1}}\bigg(\int_{y: |x-y|\geq 1} \frac{v_n^2(y)}{|x-y|^{N+2s}}\;{\rm d}y\bigg){\rm d}x\no\\
  &\leq \int_{x\in B_{k+1}}\bigg(\int_{B_R}v_n^2(y){\rm d}y+\int_{B_R^c\cap\{|x-y|\geq 1\}}\frac{v_n^2(y)}{|y|^{N+2s}}\frac{|y|^{N+2s}}{|x-y|^{N+2s}}{\rm d}y\bigg)dx\\
  &\leq\int_{x\in B_{k+1}}\bigg(o(1)+C\int_{B_R^c}\frac{v_n^2(y)}{|y|^{N+2s}}{\rm d}y\bigg)dx\no\\
    &\leq C'''\bigg(o(1)+C\big(\int_{B_R^c}{|v_n|^{2^*_s}}{\rm d}y\big)^\frac{2}{2^*_s}\big(\int_{B_R^c}\frac{{\rm d}y}{|y|^{(N+2s)N/2s}}\big)^\frac{2s}{N}\bigg)
<\eps \quad\mbox{for}\,\, R>>k+1,\no
 \end{align}
since $(v_n)_n$ is uniformly bounded in $L^{2^*_s}(\Rn)$ and $|y|^{(N+2s)N/2s}\in L^1(\{|y|>1\})$. Moreover,
 \begin{align}\lab{31-5-4}
\mathbb{I}_n^{32}&=\int_{x\in B_{k+1}}\int_{y\in B_{k+1}^c\cap\{|x-y|\leq 1\}}
\frac{|\chi (x)-\chi (y)|^2v_n^2(y)}{|x-y|^{N+2s}}\;{\rm d}y{\rm d}x\no\\
&\leq C\int_{x\in B_{k+1}}\int_{y\in B_{k+1}^c}\frac{\mathbf{1}_{|x-y|\leq 1}(x,y)v_n^2(y)}{|x-y|^{N+2s-2}}\;{\rm d}y{\rm d}x\no\\
&= C\int_{y\in B_{k+1}^c}\bigg(\int_{x\in B_{k+1}}\frac{\mathbf{1}_{|x-y|\leq 1}(x,y)}{|x-y|^{N+2s-2}}{\rm d}x\bigg) v_n^2(y){\rm d}y\\
&= C\int_{y\in B_{k+2}}\bigg(\int_{x\in B_{k+1}}\frac{\mathbf{1}_{|x-y|\leq 1}(x,y)}{|x-y|^{N+2s-2}}{\rm d}x\bigg) v_n^2(y){\rm d}y\no\\
&\leq C''''\int_{y\in B_{k+2}} v_n^2(y){\rm d}y=o(1).\no
\end{align}
Combining \eqref{31-5-1}--\eqref{31-5-4}, we obtain
$$\bigg(\displaystyle\iint_{\R^{2N}}\frac{|\chi (x)-\chi (y)|^2v_n^2(y)}{|x-y|^{N+2s}}\;{\rm d}x{\rm d}y\bigg)^\frac{1}{2}=o(1)$$
as $n\to\infty$.
Similarly, it follows that
$$\bigg(\displaystyle\iint_{\R^{2N}}\frac{|\chi (x)-\chi (y)|^2h^2(x)}{|x-y|^{N+2s}}\;{\rm d}x{\rm d}y\bigg)^\frac{1}{2}\leq \|h\|_{\Hs}.$$
Hence Claim 4 is proved.
\medskip

Therefore, using \eqref{PS8} and the fact that $\tilde f(k)=0$, we obtain as $n\to\infty$
 \begin{align*}
\prescript{}{\hms}{\langle}\dIo(\chi v_n), h{\rangle}_{\dot H^s} &= \langle \chi v_n,h\rangle_{\gamma}-\int_{\Rn}K(x)|\chi v_n|^{2_s^*-2}(\chi v_n)h\;{\rm d}x\\
    &= \int_{\Rn}K(x)\{\chi-\chi^{2_s^*-1}\}|v_n|^{2_s^*-2}v_nh\;{\rm d}x+o(\|h\|)\\
    &\leq C\|K\|_{L^{\infty}(\Rn)}\bigg(\int_{B_{k+1}\setminus B_{k}}\!\!\!|v_n|^{2_s^*}{\rm d}x\bigg)^{\tfrac{2_s^*-1}{2_s^*}}\!\!\|h\|_{\ga}+o(\|h\|)
    =o(\|h\|).
\end{align*}
This is the required inequality.    	
\medskip
	
Now, as $n\to\infty$
\begin{equation}\label{PS9}\begin{aligned}
\int_{\Rn}K(x)|v_n|^{2_s^*}{\rm d}x &= \int_{\Rn}K(x)|\chi v_n+(1-\chi) v_n|^{2_s^*}{\rm d}x\\
	 &=\int_{\Rn}K(x)|\chi v_n|^{2_s^*} {\rm d}x+\int_{\Rn}K(x)|(1-\chi)v_n|^{2_s^*}{\rm d}x+o(1).
\end{aligned}\end{equation}
The last line in \eqref{PS9} follows from the fact
that supp$(\chi)\subseteq B_{k+1}$ and supp$(1-\chi)\subset \Rn\setminus B_{k}$ and all the remaining terms in the expansion of $|\chi v_n+(1-\chi) v_n|^{2_s^*}$ involves product of some powers of $\chi v_n$ and $(1-\chi)v_n$ whose support lies in $B_{k+1}\setminus B_{k},$ but in the definition of $\chi$ we have chosen the same $k$ for which $\tilde{f}(k)=0.$
	
We know that $(v_n)_n$ is a $(PS)$ sequence of $\bar I_{K,0,0}^\ga$ at the level $\ba-\IKtf(\bar u)$. Hence, from \eqref{PS9} the level of the $(PS)$ sequence $(v_n)_n$ of $\bar I_{K,0,0}^\ga$ is integrally split between the two new $(PS)$ sequences $(\chi v_n)_n$ and $\big((1-\chi) v_n\big)_n$.
\medskip
	
Let $\mathcal{K}$ denote the Kelvin transform in $\Hs$ given by,
$$\mathcal{K}u(x):= \frac{1}{|x|^{N-2s}}u(|x|^{-2}x).$$ 
Therefore, it is known that (see \cite{RS}),
$$(-\De)^s \mathcal{K}u(x)=\frac{1}{|x|^{N+2s}}(-\De)^s u(|x|^{-2}x).$$

\noindent{\bf Claim 5:} $\|\mathcal{K}(u)\|_{\Hs}= \|u\|_{\Hs}$.
	
To prove the claim, first assume that $u\in C_0^{\infty}(\Rn)$. Thus
\be\label{PS11}
	 |(-\De)^s\mathcal{K}u(x)|\leq \frac{C}{1+|x|^{N+2s}}.
\ee
Therefore,
\begin{align*}
\|\mathcal{K}(u)\|_{\Hs} &=  \int_{\Rn}|(-\De)^{\tfrac{s}{2}}\mathcal{K}u(x)|^2\;{\rm d}x\\
	 &= \int_{\Rn}(-\De)^{s}\mathcal{K}(u(x)) \mathcal{K}u(x)\;{\rm d}x\;\;\mbox{(Using \eqref{PS11} and $\mathcal{K}(u)\in\Hs$)}\\
	 &=\int_{\Rn} \frac{1}{|x|^{N+2s}}(-\De)^s u(|x|^{-2}x)\frac{1}{|x|^{N-2s}}u(|x|^{-2}x) {\rm d}x\\
	 &=\int_{\Rn}\big((-\De)^s u(x)\big) u(x) {\rm d}x\\
	 &=\int_{\Rn}|(-\De)^{\tfrac{s}{2}}u(x)|^2 {\rm d}x \;\;\(\mbox{as }u\in C_0^{\infty}(\Rn)\)\\
	 &=\|u\|_{\Hs}^2
\end{align*}
Next for any $u\in\Hs$, let $(u_n)_n\in C_0^{\infty}(\Rn)$ be such that $u_n\to u$ in $\Hs$. Then \be\label{PS12}\|\mathcal{K}(u_n)\|_{\Hs}=\|u_n\|_{\Hs}\to \|u\|_{\Hs}.\ee Thus,
$$\|\mathcal{K}(u_n)-\mathcal{K}(u_m)\|_{\Hs}
=\|\mathcal{K}(u_n-u_m)\|_{\Hs}=\|u_n-u_m\|_{\Hs}\underset{n,m\to\infty}
{\To} 0.$$
Hence, $(\mathcal{K}(u_n))_n$  is a Cauchy sequence in $\Hs,$ so there exists $v\in\Hs$ such that $\mathcal{K}(u_n)\to v.$ Now, as $u_n\to u$ a.e. in $\Rn$ so $\mathcal{K}(u_n)\to \mathcal{K}(u)$ a.e. in $\Rn$. Consequently, $v=\mathcal{K}(u).$ Therefore, passing the limit in \eqref{PS12}, we have $\|\mathcal{K}(u)\|_{\Hs}=\|u\|_{\Hs}$ for all $u\in\Hs$.
	
	  Using Claim 5 along with standard change of variable, it is easy to see that
	  $$\bar I_{K,0,0}^\ga\big(\mathcal{K}(u)\big)=\frac{1}{2}\|u\|^2_{\dot{H}^s}-\frac{\ga}{2}\int_{\Rn}\frac{|u(x)|^2}{|x|^{2s}}{\rm d}x-\frac{1}{2^*_s}\int_{\Rn}\mathcal{K}\big(|x|^{-2}x\big)|u|^{2^*_s}(x){\rm d}x,$$
that is, $\bar I_{K,0,0}^\ga\circ\mathcal{K}$ has the same expression as $\bar I_{K,0,0}^\ga$ except that $K(x)$ has to be replaced by $K(|x|^{-2}x)$. Hence, Steps~5 and~7 
 can be applied to $(\mathcal{K}\big((1-\chi)v_n\big))_n$, since this sequence is now a $(PS)$ sequence for $\bar I_{K,0,0}^\ga\circ\mathcal{K}$ in $\dot H^s_0(B_{\frac{1}{k}})$. Using again  either Step~5 or Step~7, we obtain the characterization of $(\mathcal{K}\big((1-\chi)v_n\big))_n$ and from that we deduce the characterization of $((1-\chi) v_n)_n$; the only point which needs to be taken care of $\mathcal{K}\big(W(\tfrac{x-y_n^j}{r_n^j})\big)$. This is the concern in Lemma~\ref{L2}.
	
Finally $(vi)$ and $(vii)$ follow as in \cite[Theorem~4]{PS}. Thus the proof is completed.
\end{proof}

\section{Proof of the main Theorem \ref{th:ex-f}}\label{sec3}
In this section we assume without further mentioning that all the assumptions of
 Theorem~\ref{th:ex-f} are satisfied.
We first establish existence of two positive critical points for the  functional
	\be
	 I_{K,t,f}^{\gamma}(u)
=\frac{1}{2}\|u\|_{\gamma}^{2}-\frac{1}{2_s^*(t)}\int_{\R^N}K(x)\frac{u_+^{2_s^*(t)}}{|x|^t}\,{\rm d}x -\prescript{}{\hms}{\langle}f,u{\rangle}_{\dot{H}^s}.\no
	\ee
Clearly, if $u$ is a critical point of $I_{K,t,f}^{\ga}$ , then $u$ solves
	\begin{equation}\label{P1}\begin{cases}
	(-\Delta)^s u -\gamma\dfrac{u}{|x|^{2s}}=K(x)\dfrac{u_+^{2^*_s(t)-1}}{|x|^t}+f(x)\quad\mbox{in }\,\,\Rn,\\
	u\in \dot{H}^s(\Rn).
	\end{cases}
	\end{equation}
	
\begin{remark}\lab{r:30-7-3}
{\rm If $u$ is a weak solution of \eqref{P1} and $f$ is a nonnegative functional in $\Hs'$, then taking $v = u_{-}$ as a test function in \eqref{P1}, we obtain
$$	-\|u_-\|_{\ga}-\iint_{\R^{2N}}\frac{|u_+(y)u_-(x)+u_+(x)u_-(y)|}{|x-y|^{N+2s}}\;{\rm d}x{\rm d}y = \prescript{}{\hms}{\langle}f,u_-{\rangle}_{\dot{H}^s}\geq 0,
$$
which in turn implies that $u_{-}ˆ' \equiv 0$, i.e., $u\geq 0$. Therefore, the maximum principle \cite[Theorem~1.2]{DPQ} yields	that $u$ is a positive solution to \eqref{P1}. Hence $u$ is a solution to \eqref{P}.}
\end{remark}
To establish the existence of two critical points for $I_{K,t,f}^{\ga}$, we first need to prove some auxiliary results. Towards that, we partition $\Hs$ into three disjoint sets. Let  $\psi_t :\Hs\to\R$ be defined by
$$
\psi_t(u):=\|u\|_{\ga}^2-\big(2_s^*(t)-1\big)\|K\|_{L^{\infty}(\Rn)}\int_{\Rn}\frac{|u|^{2_s^*(t)}}{|x|^t}\;{\rm d}x
$$
and set
\begin{gather*}
\Sigma_1^t: =\big\{u\in\Hs \;:\; u=0\mbox{ or }\,\, \psi_t(u)>0\big\}, \quad \Sigma_2^t: =\big\{u\in\Hs \;:\; \psi_t(u)<0\big\},\\
\Sigma^t :=\big\{u\in\Hs \;:\;\psi_t(u)=0\big\}.\end{gather*}
\begin{remark}\lab{r:30-7-1}
{\rm If  $u\in\Sigma^t$,  then
$$\|u\|_{\ga}^2 =\big(2_s^*(t)-1\big)\|K\|_{L^{\infty}(\Rn)}\int_{\Rn}\frac{|u|^{2_s^*(t)}}{|x|^t}\;{\rm d}x\leq \big(2_s^*(t)-1\big)\|K\|_{L^{\infty}}S_{\ga,t,s}^{-\frac{2_s^*(t)}{2}}\|u\|_{\ga}^{2_s^*(t)}.$$
	Therefore, $\|u\|_{\ga}$ and $\|u\|_{L^{2_s^*(t)}(\Rn, |x|^{-t})}$ are bounded away from $0$ for all $u\in\Si^t$.}
\end{remark}
Set
\be\label{S31}
c_{0}^t:=  \inf_{\Sigma_1^t} I_{K,t,f}^\ga(u),\quad c_1^t:=\inf_{\Sigma^t}I_{K,t,f}^\ga(u), \quad t\geq 0.
\ee

\begin{remark}\lab{r:30-7-2}
{\rm For any $\la>0$ and $u\in\Hs$	
$$
\psi_t(\la u)=\la^2\|u\|_{\ga}^2-\la^{2_s^*(t)}
\big(2_s^*(t)-1\big)\|K\|_{L^{\infty}(\Rn)}\int_{\Rn}\frac{|u|^{2_s^*(t)}}{|x|^t}\;{\rm d}x.$$
 Moreover, $\psi_t(0)=0$ and $\la\mapsto \psi_t(\la u)$ is a strictly concave function. Thus for any $u\in\Hs$ with $\|u\|_{\ga}=1$, there exists a unique $\la=\la(u)$ such that $\la u\in\Sigma^t$. Moreover, as $\psi_t(\la u)=\big(\la^2-\la^{2_s^*(t)}\big)\|u\|_{\ga}^2$ for all $u\in\Sigma^t$, then $\la u\in\Sigma_1^t$ for all $\la\in (0,1)$
and $\la u\in\Sigma_2^t$ for all $\la>1$.}
\end{remark}

\begin{lemma}\lab{l31}
Assume that $C_t$ is defined as in Theorem~$\ref{th:ex-f}$. Then
$$
\frac{4s-2t}{N-2t+2s}\|u\|_{\ga}\geq C_t S_{\ga,t,s}^\frac{N-t}{4s-2t}\quad\mbox{ for all }\,\,u\in\Sigma^t, \,\, t\geq 0.$$
\end{lemma}

\begin{proof}
Fix  $u\in\Sigma^t$. Then
$$ \bigg(\int_{\Rn}\frac{|u|^{2_s^*(t)}}{|x|^t}\;{\rm d}x\bigg)^{\frac{1}{2_s^*(t)}}=\frac{\|u\|_{\ga}^{\frac{2}{2_s^*(t)}}}{\bigg(\big(2_s^*(t)-1\big)\|K\|_{L^{\infty}(\Rn)}\bigg)^{\frac{1}{2_s^*(t)}}}.$$	
	Combining this with the definition of $S_{\ga,t,s}$ yields
$$\|u\|_{\ga} \geq S_{\ga,t,s}^{\frac{1}{2}}\bigg(\int_{\Rn}\frac{|u|^{2_s^*(t)}}{|x|^t}\;{\rm d}x\bigg)^{\frac{1}{2_s^*(t)}}=S_{\ga,t,s}^{\frac{1}{2}}\frac{\|u\|_{\ga}^{\frac{2}{2_s^*(t)}}}{\bigg(\big(2_s^*(t)-1\big)\|K\|_{L^{\infty}(\Rn)}\bigg)^{\frac{1}{2_s^*(t)}}} $$
for all $u\in\Sigma^t$.
From here using the definition of $C_t$, we conclude the
proof of the lemma.
\end{proof}

\begin{lemma}\lab{l:32}
	Assume that  $t\geq 0$, $C_t$ is given as in Theorem~$\ref{th:ex-f}$ and $c_{0}^t$, $c_1^t$ are defined as in \eqref{S31}. Further if
\be\label{S32}
\inf_{\stackrel{u\in\Hs}{\|u\|_{L^{2^*_s(t)}(\Rn, |x|^{-t})}}=1}\bigg\{C_t\|u\|_{\ga}^{\frac{N-2t+2s}{2s-t}}-\prescript{}{\hms}{\langle}f,u{\rangle}_{\dot{H}^s}\bigg\}>0,
	\ee
then $c_{0}^t<c_1^t$.
\end{lemma}

\begin{proof}
Define
\be\label{S36}
\tilde{J}_{t}(u):=\frac{1}{2}\|u\|_{\ga}^2-
\frac{\|K\|_{L^{\infty}}}{2_s^*(t)}\int_{\Rn}\frac{|u|^{2_s^*(t)}}{|x|^t}\;{\rm d}x-\prescript{}{\hms}{\langle}f,u{\rangle}_{\dot{H}^s}.
\ee
\underline{\bf Step I:} In this step we prove that there exists $\beta_t >0$ such that
$$\frac{d}{dp}\tilde{J}_{t}(pu)\bigg{|}_{p=1}\geq \beta_t \quad\mbox{for all }\,\, u\in\Sigma^t.$$
Indeed, using the definition of $\Si^t$ and the value of $C_t$,  we have for  $u\in \Si^t$
\begin{align}\label{S37}
\frac{d}{dp}\tilde{J}_{t}(pu)\bigg{|}_{p=1}\!\!\!&=\|u\|_{\ga}^2-\|K\|_{L^{\infty}(\Rn)}\int_{\Rn}\frac{|u|^{2_s^*(t)}}{|x|^t}\;{\rm d}x-\prescript{}{\hms}{\langle}f,u{\rangle}_{\dot{H}^s}\no\\
&=\bigg(1-\frac{1}{2_s^*(t)-1}\bigg)\|u\|_{\ga}^2
-\prescript{}{\hms}{\langle}f,u{\rangle}_{\dot{H}^s}\\	&=\frac{4s-2t}{N-2t+2s}\|u\|_{\ga}^2
-\prescript{}{\hms}{\langle}f,u{\rangle}_{\dot{H}^s}
= C_{t}\frac{\|u\|_{\ga}^\frac{N+2s-2t}{2s-t}}{\|u\|_{L^{2^*_s(t)}(\Rn, |x|^{-t})}^\frac{N-t}{2s-t}}-\prescript{}{\hms}{\langle}f,u{\rangle}_{\dot{H}^s}.\no
\end{align}
Furthermore, \eqref{S32} implies that there exists $d>0$ such that
\be\label{S38}
\inf_{\stackrel{u\in\Hs}{\|u\|_{L^{2^*_s(t)}(\Rn, |x|^{-t})}=1}}\bigg\{C_t\|u\|_{\ga}^{\frac{N-2t+2s}{2s-t}}
-\prescript{}{\hms}{\langle}f,u{\rangle}_{\dot{H}^s}\bigg\}\geq d.
\ee
	Observe that,
	\bea
\eqref{S38} &\iff& C_{t}\frac{\|u\|_{\ga}^{\frac{N+2s-2t}{2s-t}}}{\|u\|_{L^{2^*_s(t)}(\Rn, |x|^{-t})}^\frac{N-t}{2s-t}}-\prescript{}{\hms}{\langle}f,u{\rangle}_{\dot{H}^s}\geq d,\quad\int_{\Rn}\frac{|u|^{2_s^*(t)}}{|x|^t}\;{\rm d}x=1\no\\
	&\iff& C_{t}\frac{\|u\|_{\ga}^{\frac{N+2s-2t}{2s-t}}}{\|u\|_{L^{2^*_s(t)}(\Rn, |x|^{-t})}^\frac{N-t}{2s-t}}-\prescript{}{\hms}{\langle}f,u{\rangle}_{\dot{H}^s}\geq d\|u\|_{L^{2^*_s(t)}(\Rn, |x|^{-t})}, \quad u\in\Hs\setminus\{0\}.\no
	\eea
	Hence, plugging back the above estimate into \eqref{S37} and using Remark~\ref{r:30-7-1}, we complete the proof of Step I.
\medskip
	
\noindent\underline{\bf Step II:} Let $(u_n^t)_n$ be a minimizing sequence for $I_{K,t,f}^\ga$ on $\Sigma^t$, that is,
$$I_{K,t,f}^\ga(u_n^t)\to c_1^t\quad\mbox{and}\quad \|u_n^t\|_{\ga}^2=\|K\|_{L^{\infty}(\Rn)}\big(2_s^*(t)-1\big)
\displaystyle\int_{\Rn}\frac{|u_n^t|^{2_s^*(t)}}{|x|^t}{\rm d}x.$$
Therefore,
\begin{align*}
c_1^t+o(1)= I_{K,t,f}^\ga(u_n)\geq \tilde{J}_{t}(u_n^t)\geq \bigg(\frac{1}{2}-\frac{1}{2_s^*(t)\big(2_s^*(t)-1\big)}
\bigg)\|u_n^t\|_{\ga}^2-\|f\|_{\hms}\|u_n\|_{\ga}.
\end{align*}
This implies that $(\tilde{J}_{t}(u_n^t))_n$ is bounded and  $(\|u_n^t\|_{\ga})_n$, $(\|u_n^t\|_{L^{2_s^*(t)}(\Rn, |x|^{-t})})_n$ are bounded.\smallskip
	
\noindent\underline{\bf Claim:}  $c_0^t<0$ for all $t\geq 0$.
	
	To prove this claim, it is enough to show that there exists $v^t\in \Sigma_1^t$ such that $I_{K,t,f}^\ga(v^t)<0.$ Note that, thanks to Remark~\ref{r:30-7-2}, we can choose $u^t\in\Sigma^t$ such that $\prescript{}{\hms}{\langle}f,u^t{\rangle}_{\dot{H}^s}>0.$
	
	Therefore,
$$I_{K,t,f}^\ga(pu^t)= p^2\bigg[\frac{\big(2_s^*(t)-1\big)\|K\|_{L^{\infty}(\Rn)}}{2}-\frac{p^{2_s^*(t)-2}}{2_s^*(t)}\bigg]\int_{\Rn}\frac{|u^t|^{2_s^*(t)}}{|x|^t}\;{\rm d}x - p\prescript{}{\hms}{\langle}f,u^t{\rangle}_{\dot{H}^s}<0$$
	for $p<< 1$. Moreover, $pu^t\in\Sigma_1^t$ by Remark~\ref{r:30-7-2}. Hence the claim follows.
	
Thanks to the above claim, $I_{K,t,f}^\ga(u_n^t)<0$  for large $n$. Consequently,
\begin{align*}
0>I_{K,t,f}^\ga(u_n^t)\geq \bigg(\frac{1}{2}-\frac{1}{2_s^*(t)\big(2_s^*(t)-1\big)}\bigg)
\|u_n^t\|_{\ga}^2-\prescript{}{\hms}{\langle}f,u_n^t{\rangle}_{\dot{H}^s}
\end{align*}
for large $n$. This in turn implies that $\prescript{}{\hms}{\langle}f,u_n^t{\rangle}_{\dot{H}^s}>0$ for $n$ large enough. Hence, $\frac{d}{dp}\tilde{J}_{t}(pu_n^t)<0$ for $p>0$ small enough. Thus, by Step I  there exists $p_n^t\in (0,1)$ such that $\frac{d}{dp}\tilde{J}_{t}(p_n^tu_n^t)=0$.
	
Moreover,  it is easy to check that for all $u^t\in \Si^t$, the map $p\mapsto \frac{d}{dp}\tilde{J}_{t}(pu^t)$ is strictly increasing in $[0,1)$ and therefore, we can conclude that $p_n^t$ is unique.
\medskip
	
\noindent\underline{\bf Step III:}  In this step we show that
\be\label{S39} \liminf_{n\to\infty}\bigg\{\tilde{J}_{t}(u_n^t)-\tilde{J}_{t}(p_n^tu_n^t)\bigg\}>0.
\ee
We observe that $\tilde{J}_{t}(u_n^t)-\tilde{J}_{t}(p_n^tu_n^t)
=\displaystyle\int_{p_n^t}^{1}\frac{d}{dp}\tilde{J}_{t}(pu_n){\rm d}p$ and  that for all $n\in\mathbb{N}$ there exists $\xi_n^t>0$ such that $p_n^t\in(0,1-2\xi_n^t)$ and $\frac{d}{dp}\tilde{J}_{t}(pu_n^t)\geq \frac{\beta_t}{2}$ for $p\in [1-\xi_n^t,1].$

To establish \eqref{S39}, it is enough to show that $\xi_n^t>0$ can be chosen independently of $n\in\mathbb{N}$. This is possible, since $\frac{d}{dp}\tilde{J}_{t}(pu_n^t)\bigg{|}_{p=1}\geq\beta_t$, and   $(u_n^t)_n$ is bounded, so that for all $n$ and $p\in [0,1]$
\begin{align*}
\bigg{|}\frac{d^2}{dp^2}\tilde{J}_{t}(pu_n^t)\bigg{|}&= \bigg{|}\|u_n^t\|_{\ga}^2-\big(2_s^*(t)-1\big)\|K\|_{L^{\infty}}p^{2_s^*(t)-2}\int_{\Rn}\frac{|u_n^t|^{2_s^*(t)}}{|x|^t}\;{\rm d}x\bigg{|}\\
&=\bigg{|}\big(1-p^{2_s^*(t)-2}\big)\|u_n^t\|_{\ga}^2\bigg{|}\leq C.
\end{align*}

\noindent\underline{\bf Step IV:}  From the definition of $I_{K,t,f}^\ga$ and $\tilde{J}_{t}$, it immediately follows that $\frac{d}{dp}I_{K,t,f}^\ga(pu)\geq \frac{d}{dp}\tilde{J}_{t}(pu)$ for all $u\in\Hs$ and for all $p>0$. Hence,
$$I_{K,t,f}^\ga(u_n^t)-I_{K,t,f}^\ga(p_n^tu_n^t)=\int_{p_n^t}^{1}\frac{d}{dp}I_{K,t,f}^\ga(pu_n^t)\;{\rm d}p\geq \int_{p_n^t}^{1}\frac{d}{dp}\tilde{J}_{t}(pu_n^t)\;{\rm d}p = \tilde{J}_{t}(u_n^t)-\tilde{J}_{t}(p_n^tu_n^t).$$
Since $(u_n^t)_n\subset \Sigma^t$ is a minimizing sequence for $I_{K,t,f}^\ga$ on $\Sigma^t$ and $p_n^tu_n^t\in\Sigma_1^t$, then by \eqref{S39}
$$c_0^t= \inf_{\Sigma_1^t}I_{K,,t,f}^\ga(u)<\inf_{\Sigma^t}I_{K,t,f}^\ga(u)=c_1^t.$$
\end{proof}

\begin{proposition}\label{p:30-7-1}
	Assume that $t\geq 0$  and \eqref{S32} holds. Then $I_{K,t,f}^{\ga}$ has a critical point $u_t\in \Sigma_1^t$ with $I_{K,t,f}^{\ga}(u_t)=c_0^t.$ In particular, $u_t$ is a positive solution to \eqref{P}.
\end{proposition}

\begin{proof}
	We divide the proof in few steps.

\noindent\underline{\bf Step 1:} In this step we show that $c_0^t>-\infty.$
	
	From the definition of $\tilde{J}_{t}$ in \eqref{S36},  we have $I_{K,t,f}^\ga(u)>\tilde{J}_{t}(u)$. Therefore, in order to prove Step 1, it is enough to show that $\tilde{J}_{t}$ is bounded from below. From
the definition of $\Sigma_1^t$,
\be\label{S310}
\tilde{J}_{t}(u)\geq \bigg(\frac{1}{2}-\frac{1}{2_s^*(t)\big(2_s^*(t)-1\big)}\bigg)
\|u\|_{\ga}^2-\|f\|_{\hms}\|u\|_{\ga}\quad\mbox{for all }\, u\in\Sigma_1^t.
	\ee
	As the RHS is a quadratic function in $\|u\|_{\ga}$, then $\tilde{J}_{t}$ is bounded from below and thus so is $I_{K,t,f}^\ga.$
\medskip
	
\noindent\underline{\bf Step 2:}  In this step we show that there exists a bounded nonnegative $(PS)$ sequence $(u_n^t)_n\subset \Sigma_1^t$ for $I_{K,t,f}^\ga$ at the level $c_{0}^t$. Let $(u_n^t)_n\subset\bar{\Sigma_1^t}$ such that $I_{K,t,f}^\ga(u_n^t)\to c_{0}^t$. Since Lemma~\ref{l:32} implies $c_0^t<c_1^t$,  without any restriction we can assume that $(u_n)_n\subset \Sigma_1^t$. Further, using Ekeland's variational principle, $(u_n^t)_n$ admits a  $(PS)$ subsequence, still called $(u_n^t)_n$, in $\Sigma_1^t$ for $I_{K,t,f}^\ga$ at the level $c_{0}^t$.   Moreover, as $I_{K,t,f}^\ga(u)\geq \tilde{J}_{t}(u)$, from \eqref{S310} it follows that $(u_n^t)_n$ is a bounded sequence in $\Hs$. Therefore, up to a subsequence, $u_n^t\rightharpoonup u_t$ in $\Hs$ and $u_n^t\to u_t$ a.e. in $\Rn$. In particular, $(u_n^t)_{+}\to (u_t)_{+}$ and $(u_n^t)_{-}\to (u_t)_{-}$ a.e. in $\Rn$.	Moreover, the fact that $f$ is a nonnegative functional
gives as $n\to\infty$
\begin{align*}
o(1)&=   \prescript{}{\hms}{\big\langle}(I_{K,t,f}^{\ga})'(u_n^t), (u_n^t)_{-}{\big\rangle}_{\dot H^s}\\
&=  \langle u_n^t,(u_n^t)_{-}\rangle_{\ga} -\int_{\Rn}\frac{K(x)(u_n^t)_{+}^{2_s^*(t)-1}(u_n^t)_{-}}{|x|^t}\;{\rm d}x -\prescript{}{\hms}{\langle}f, (u_n^t)_{-}{\rangle}_{\dot H^s}\\
&\leq  -\|(u_n^t)_{-}\|_{\ga}-\iint_{\R^{2N}}\frac{(u_n^t)_{-}(x)(u_n^t)_{+}(y)+(u_n^t)_{+}(x)(u_n^t)_{-}(y)}{|x-y|^{N+2s}}\;{\rm d}x{\rm d}y\\
&\leq -\|(u_n^t)_{-}\|_{\ga}^2.
\end{align*}
This implies that $(u_n)_{-}\to 0$ in $\Hs$ and so $(u_n)_{-}\to 0$ in a.e. in $\Rn$, which in turn yields that $(u_0)_{-}\equiv 0$, that is,  $u_0\geq 0$ a.e. in $\Rn$.
 Consequently, without loss of generality, we can assume $(u_n^t)_n$ is a nonnegative $(PS)$ sequence. This completes the proof of Step~2.
\medskip
	
\noindent\underline{\bf  Step 3:}  In this step we show that $u_n^t\to u_t$ in $\Hs$.

Applying Theorem~\ref{th:PS}, we get as $n\to\infty$
\be\lab{J5}
u_n^t= u_t+\sum_{j=1}^{n_1}K(y^j)^{-\frac{N-2s}{4s}}W^{r_n^j,\;y_n^j}+\sum_{k=1}^{n_2}(W_{\gamma,t}^k)^{R_n^k,0}+o(1), \quad\mbox{if }t=0,
\ee
and
\be\lab{J5"}
u_n^t= u_t+\sum_{k=1}^{n_2}(W_{\gamma,t}^k)^{R_n^k,0}+o(1),\quad\mbox{if }t>0,
\ee
where $(I_{K,t,f}^{\ga})'(u_t)=0$, $W$ is the unique positive solution of \eqref{W}, where $W_{\gamma,t}^k$, $k=1,2, \cdots n_2$ are positive ground state solutions of \eqref{Q}. Moreover, $(y_n^j)_n$,  $(r_n^j)_n$ and $(R_n^k)_n$ are some appropriate sequences with $R_n^k\to 0$ for each $k=1,\cdots n_2$,  $r_n^j\to 0$, $\frac{r_n^j}{y_n^j}\to 0$ and
either $y_n^j\to y^j$ or $|y_n^j|\to\infty$ and  for all $j=1,\cdots n_1$,
are appropriate sequences. To prove Step~3, we need to show that $n_1=0=n_2$. We prove this by the method of contradiction.

Suppose $t=0$. The case $t>0$ is comparatively easier and the proof of
that case will easily follow from arguments that we present in the case of $t=0$. Also for $t>0$, one can argue as in \cite[Proposition~3.1]{BCG}.
\vspace{2mm}

Thus, let us assume that $t=0$ and $u_n^t \not\to u_t$ in $\Hs$. For simplicity of notations, we denote $u_n^0$ by $u_n$. 
\vspace{2mm}

Then either $n_1\neq 0$ or $n_2\neq 0$ or both $n_1,\, n_2\neq 0$ in \eqref{J5}. Here we prove the last case that is when $n_1$ and $n_2$ both are non zero. If one of them is zero, that case is  again comparatively easier and argument in that case will follow from this case. First we observe that
\begin{align*}\lab{12-4-3}
\psi_0\bigg(K(y^j)^{-\frac{N-2s}{4s}}W^{r_n^j, y_n^j}\bigg)&=K(y^j)^{-\frac{N-2s}{2s}}\|W\|^2_{\ga}-(2^*_s-1)\|K\|_{L^\infty(\Rn)}K(y^j)^{-\frac{N}{2s}}
\|W\|_{L^{2^*_s}(\R^N)}^{2^*_s}\no\\
&=K(y^j)^{-\frac{N}{2s}}\bigg(K(y^j)-(2^*_s-1)
\|K\|_{L^\infty(\Rn)}\bigg)\|W\|_{L^{2^*_s}(\R^N)}^{2^*_s}\\
&\qquad\qquad -\ga K(y^j)^{-\frac{N-2s}{2s}}\int_{\Rn}\frac{|W|^2}{|x|^{2s}}{\rm d}x <0.\no
\end{align*}				
Similarly,
\begin{equation}\label{5-6-2}\begin{aligned}
\psi_0\big((W_{\gamma,0}^k)^{R_n^k,0}\big)=\psi_0(W_{\gamma,0}^k)&=\|W_{\gamma,0}^k\|_{\ga}^2-(2^*_s-1)\|K\|_{L^\infty(\Rn)}\|W_{\gamma,0}^k\|_{L^{2^*_s}(\R^N)}^{2^*_s}\no\\
&=\big(1-(2^*_s-1)\|K\|_{L^\infty(\Rn)}\big)\|W_{\gamma,0}^k\|_{\ga}^2<0.	
\end{aligned}\end{equation}	
Theorem~\ref{th:PS} gives
$$o(1)+c_0^0=I_{K,0,f}^{\ga}(u_n)\rightarrow  I_{K,0,f}^\ga(u_0)+\sum_{j=1}^{n_1}K(y^j)^{-\frac{N-2s}{2s}} I_{1,0,0}^{0}(W)+\sum_{k=1}^{n_2}I_{1,0,0}^{\gamma}(W_{\gamma,0}^{k}).$$
As $K>0$, $I_{1,0,0}^{0}(W)=\frac{s}{N}\|W\|^2_{\dot{H}^s}>0$
and $I_{1,0,0}^{\ga}(W_{\ga,0})=\frac{s}{N}\|W_{\ga,0}\|^2_{\ga}>0$, from the above expression we obtain $ I_{K,0,f}^\ga(u_0)<c_0^0$. This in turn yields $u_0\not\in \Si_1^0$ and
\be\lab{12-4-4} \psi_0(u_0)\leq 0.\ee 	
	Next, we evaluate $\psi_0\bigg(u_0 +\displaystyle\sum_{j=1}^{n_1}K(y^j)^{-\frac{N-2s}{4s}}W^{r_n^j,\;y_n^j}+\sum_{k=1}^{n_2}(W_{\gamma,0}^k)^{R_n^k,0}\bigg)$. Since $u_n\in \Si_1^0$, we have $\psi_0(u_n)\geq 0$. Therefore, the uniform continuity of $\psi_0$ and \eqref{J5}  imply
\be\label{J8}
0\leq \liminf_{n\rightarrow\infty}\psi_0(u_n)=\liminf_{n\rightarrow\infty} \psi_0\bigg(u_0 +\sum_{j=1}^{n_1}K(y^j)^{-\frac{N-2s}{4s}}W^{r_n^j, y_n^j}+\sum_{k=1}^{n_2}(W_{\gamma,0}^k)^{R_n^k,0}\bigg).
\ee
Since $u_0, \, W,\, {W_{\ga,0}^k}  \geq 0$  for all $k=1,\cdots, n_1$,
\begin{align}\lab{6-6-1}
 &\psi_0\bigg(u_0 +\sum_{j=1}^{n_1}K(y^j)^{-\frac{N-2s}{4s}}W^{r_n^j, y_n^j}+\sum_{k=1}^{n_2}(W_{\gamma,0}^k)^{R_n^k,0}\bigg)\no\\
 &\leq \|u_0\|^2_{\ga}+
 \sum_{k=1}^{n_2}\|W_{\gamma,0}^k\|^2_{\ga}
+\sum_{j=1}^{n_1}\|K(y^j)^{-\frac{N-2s}{4s}}W^{r_n^j, y_n^j}\|^2_{\ga}+ 2\big\langle u_0, \sum_{k=1}^{n_2}(W_{\gamma,0}^k)^{R_n^k,0} \big\rangle_{\ga} \no\\
&\quad+2\big\langle u_0, \sum_{j=1}^{n_1}K(y^j)^{-\frac{N-2s}{4s}}W^{r_n^j, y_n^j}\big\rangle_{\ga}+ 2\big\langle\sum_{k=1}^{n_2}(W_{\gamma,0}^k)^{R_n^k,0}, \sum_{j=1}^{n_1}K(y^j)^{-\frac{N-2s}{4s}}W^{r_n^j, y_n^j}\big\rangle_{\ga}\no\\
&\quad+\sum_{i, k=1, \, i\neq k}^{n_2}\big\langle(W_{\gamma,0}^i)^{R_n^i,0}, (W_{\gamma,0}^k)^{R_n^k,0}\big\rangle_{\ga}+\sum_{l, j=1, \, l\neq j}^{n_1}\big\langle K(y^l)^{-\frac{N-2s}{4s}}W^{r_n^l, y_n^l}, K(y^j)^{-\frac{N-2s}{4s}}W^{r_n^j, y_n^j}\big\rangle_{\ga}\no\\
&\quad-(2^*_s-1)\|K\|_{L^\infty(\mathbb R^N)}\bigg(\|u_0\|_{2^*_s}^{2^*_s}
+\sum_{j=1}^{n_1}\|K(y^j)^{-\frac{N-2s}{4s}}\big(W^{r_n^j, y_n^j}\big)\|_{2^*_s}^{2^*_s}+\sum_{k=1}^{n_2}
\|W_{\ga,0}^k\|_{2^*_s}^{2^*_s}\bigg)\no\\
 &\leq\psi_0(u_0)+\sum_{k=1}^{n_2}\psi_0(W_{\gamma,0}^k)+\sum_{j=1}^{n_1}\psi_0\big(K(y^j)^{-\frac{N-2s}{4s}}W^{r_n^j, y_n^j}\big)+\mbox{the above inner products}.
\end{align}	
We now prove that all the five inner products in the RHS of \eqref{6-6-1} approaches $0$ as $n\to\infty$. As $r_n^j\to 0$ and $\frac{|y_n^j|}{|r_n^j|}\to\infty$, it follows that $W^{r_n^j, y_n^j}\deb 0$ in $\Hs$ (see \cite[Lemma~3]{PS}) and $W^{r_n^j, y_n^j}\to 0$ a.e. in $\Rn$.
Choosing $R>0$ large enough as $n\to\infty$
\begin{align*}
\int_{\Rn}\frac{u_0W^{r_n^j, y_n^j}}{|x|^{2s}}{\rm d}x &\leq \int_{B_R}\frac{u_0W^{r_n^j, y_n^j}}{|x|^{2s}}{\rm d}x+\int_{|x|>R}\frac{u_0W^{r_n^j, y_n^j}}{|x|^{2s}}{\rm d}x\\
&\leq \int_{B_R}\frac{u_0W^{r_n^j, y_n^j}}{|x|^{2s}}{\rm d}x+\int_{|x|>R}\bigg(\frac{|u_0|^2}{|x|^{2s}}{\rm d}x\bigg)^\frac{1}{2}\int_{|x|>R}\bigg(\frac{|W|^2}{|x+\frac{y_n^j}{r_n^j}|^{2s}}{\rm d}x\bigg)^\frac{1}{2}\\
&=o(1),
\end{align*}
where in the 1st integral we have passed the limit using Vitali's convergence theorem via the H\"{o}lder inequality, while in the 2nd integral  simply using the Hardy inequality.
Therefore, as $n\to\infty$
\be\lab{6-6-2}
\big\langle u_0 ,\, K(y^j)^{-\frac{N-2s}{4s}}W^{r_n^j, y_n^j}\big\rangle_{\ga}=
K(y^j)^{-\frac{N-2s}{4s}}\bigg[\big\langle u_0 ,\, W^{r_n^j, y_n^j} \big\rangle_{\dot{H}^s}-\ga\int_{\Rn}\frac{u_0W^{r_n^j, y_n^j}}{|x|^{2s}}{\rm d}x\bigg]=o(1).
\ee
%
%
%
%
%
Since $R_n^k\to 0$ as $n\to\infty$, similarly we also have
\be\lab{6-6-3}
\big\langle u_0, \sum_{k=1}^{n_2}(W_{\gamma,0}^k)^{R_n^k,0} \big\rangle_{\ga}=o(1).
\ee
Now,
\begin{align*}
&\bigg\langle K(y^l)^{-\frac{N-2s}{4s}}\big(W^{r_n^l, y_n^l}\big) ,K(y^j)^{-\frac{N-2s}{4s}}\big(W^{r_n^j, y_n^j}\big) \bigg\rangle_{\ga}\\
&\;=K(y^l)^{-\frac{N-2s}{4s}}K(y^j)^{-\frac{N-2s}{4s}}(r_n^l)^{-\frac{N-2s}{2}}(r_n^j)^{-\frac{N-2s}{2}}\no\\
&\qquad\times\bigg[\iint_{\R^{2N}}\frac{\big(W(\frac{x-y_n^l}{r_n^l})-W(\frac{y-y_n^l}{r_n^l})\big)\big(W(\frac{x-y_n^j}{r_n^j})-W(\frac{y-y_n^j}{r_n^j})\big)}{|x-y|^{N+2s}}{\rm d}x{\rm d}y\\
&\qquad\qquad\qquad\qquad\quad-\ga\int_{\Rn}\frac{W(\frac{x-y_n^l}{r_n^l})W(\frac{x-y_n^j}{r_n^j})}{|x|^{2s}}{\rm d}x\bigg]\\
&\;=K(y^l)^{-\frac{N-2s}{4s}}K(y^j)^{-\frac{N-2s}{4s}}(r_n^l)^\frac{N-2s}{2}(r_n^j)^{-\frac{N-2s}{2}}\no\\
&\qquad\times\bigg[\iint_{\R^{2N}}\frac{\big(W(x)-W(y)\big)\big(W(\frac{r_n^lx+y_n^l-y_n^j}{r_n^j})-W(\frac{r_n^ly+y_n^l-y_n^j}{r_n^j})\big)}{|x-y|^{N+2s}}{\rm d}x{\rm d}y\\
&\qquad\qquad\qquad\qquad\quad-\ga\int_{\Rn}\frac{W(x)W(\frac{r_n^ly+y_n^l-y_n^j}{r_n^j})}{|x+\frac{y_n^l}{r_n^l}|^{2s}}{\rm d}x\bigg]\\
&=K(y^l)^{-\frac{N-2s}{4s}}K(y^j)^{-\frac{N-2s}{4s}}\bigg[\big\langle W, W_n\big\rangle_{\Hs} -\ga\int_{\Rn}\frac{WW_n}{|x+\frac{y_n^l}{r_n^l}|^{2s}}{\rm d}x\bigg],
\end{align*}
where $W_n:=(\frac{r_n^l}{r_n^j})^\frac{N-2s}{2}W\big(\frac{r_n^l}{r_n^j}x+\frac{y_n^l-y_n^j}{r_n^j}\big)$. Theorem~\ref{th:PS} $(vi)$ yields
$$\bigg|\log\big(\frac{r^l_n}{r_n^j}\big)\bigg|+\bigg|\frac{y_n^l-y_n^j}{r_n^j}\bigg|\To\infty.$$
Thus $W_n\deb 0$ in $\Hs$ (see \cite[Lemma~3]{PS}). Hence,
as $n\to\infty$
\be\lab{6-6-4}
\bigg\langle K(y^l)^{-\frac{N-2s}{4s}}\big(W^{r_n^l, y_n^l}\big) ,K(y^j)^{-\frac{N-2s}{4s}}\big(W^{r_n^j, y_n^j}\big)\bigg\rangle_{\ga}=o(1).
\ee
Similarly,
\be\lab{6-6-5}
\big\langle(W_{\gamma,0}^i)^{R_n^i,0}, (W_{\gamma,0}^k)^{R_n^k,0}\big\rangle_{\ga}=o(1)
\ee
as $|\log \frac{R_n^j}{R_n^k}|\to\infty$.

Finally, we estimate $\big\langle(W_{\gamma,0}^k)^{R_n^k,0}, K(y^j)^{-\frac{N-2s}{4s}}W^{r_n^j, y_n^j}\big\rangle_{\ga}$. First we note that  $|\log \frac{R_n^j}{R_n^k}|\to\infty$ implies that either $\frac{R_n^j}{R_n^k}\to 0$ or $\frac{R_n^j}{R_n^k}\to \infty$. Suppose $\frac{R_n^j}{R_n^k}\to 0$. Then
\begin{align*}
\big\langle(W_{\gamma,0}^k)^{R_n^k,0}, K(y^j)^{-\frac{N-2s}{4s}}W^{r_n^j, y_n^j}\big\rangle_{\ga}
&=K(y^j)^{-\frac{N-2s}{4s}}(R_n^k)^\frac{N-2s}{2}(r_n^j)^{-\frac{N-2s}{2}}\cdot\\
&\;\;\times\bigg[\iint_{\R^{2N}}\!\!\!\!\frac{\big(W^k_{\ga,0}(x)-W^k_{\ga,0}(y)\big)
\big(W(\frac{R_n^kx-y_n^j}{r_n^j})
-W(\frac{R_n^ky-y_n^j}{r_n^j})\big)}{|x-y|^{N+2s}}{\rm d}x{\rm d}y\\
&\qquad\qquad\qquad\qquad\quad-\ga\int_{\Rn}\frac{W^k_{\ga,0}(x)
W(\frac{R_n^kx-y_n^j}{r_n^j})}{|x|^{2s}}{\rm d}x\bigg]\\
&=K(y^j)^{-\frac{N-2s}{4s}}\bigg[\big\langle W^k_{\ga,0}, W^n\big\rangle_{\Hs} -\ga\int_{\Rn}\frac{W^k_{\ga,0}W^n}{|x|^{2s}}{\rm d}x\bigg],
\end{align*}
where $W^n:=(\frac{r_n^j}{R_n^k})^{-\frac{N-2s}{2}}
W\bigg(\frac{x-\frac{y_n^j}{R_n^k}}{r_n^j/R_n^k}\bigg)$. The proof of Theorem~\ref{th:PS} gives $\frac{r_n^j}{R_n^k}=\frac{s_n^jR_n^j}{R_n^k}$
for any $j$ and $k$. As $s_n^j\to 0$ and $\frac{R_n^j}{R_n^k}\to 0$, we have $\frac{r_n^j}{R_n^k}\to 0$. Moreover, $\frac{|y_n^j|}{r_n^j}\to\infty$ implies that $\frac{|y_n^j/ R_n^k|}{r_n^j/R_n^k}\to\infty$. Thus $|\log \frac{r_n^j}{R_n^k}|+|y_n^j/ R_n^k|\to\infty$. Consequently by \cite[Lemma 3]{PS}, $W^n\deb 0$ in $\Hs$. Hence, an argument similar to \eqref{6-6-2} yields
\be\lab{6-6-6}
\big\langle(W_{\gamma,0}^k)^{R_n^k,0}, K(y^j)^{-\frac{N-2s}{4s}}W^{r_n^j, y_n^j}\big\rangle_{\ga}=o(1).
\ee
On the other hand, if $\frac{R_n^j}{R_n^k}\to \infty$ then $\frac{R_n^k}{R_n^j}\to 0$. Then similarly, we also show that
$$\big\langle(W_{\gamma,0}^k)^{R_n^k,0}, K(y^j)^{-\frac{N-2s}{4s}}W^{r_n^j, y_n^j}\big\rangle_{\Hs}=K(y^j)^{-\frac{N-2s}{4s}}\langle W^n_\ga, W\rangle,$$
where $W^n_\ga(x)=\big(\frac{R_n^k}{R_n^j}\big)^{-\frac{N-2s}{2s}}W_{\ga,0}^k\bigg(\frac{x-\frac{y^j_n}{r_n^j}}{R^k_n/r_n^j}\bigg)$. Since $\frac{R_n^k}{R_n^j}\to 0$ and $\frac{|y^j_n|}{r_n^j}\to \infty$, again applying  \cite[Lemma~3]{PS}, we get $W_{\ga}^n\deb 0$ in $\Hs$. Hence, in any case \eqref{6-6-6} holds.

Combining \eqref{6-6-2}--\eqref{6-6-6} along with \eqref{6-6-1}, we have
$$\psi_0\bigg(u_0 +\sum_{j=1}^{n_1}K(y^j)^{-\frac{N-2s}{4s}}W^{r_n^j, y_n^j}+\sum_{k=1}^{n_2}(W_{\gamma,0}^k)^{R_n^k,0}\bigg)<0.$$
This contradicts \eqref{J8}. Therefore, $n_1=0$ and $n_2=0$ in \eqref{J5}. Hence, $u_n\to u_0$ in $\Hs$. Consequently, $\psi_0(u_n)\to \psi_0(u_0)$, which in turn implies that $u_0\in \bar \Si_1^t$. But, since $c_0^0<c_1^0$, we conclude$u_0\in \Si_1^t$.    Hence Step~3 follows.
\end{proof}

\begin{proposition}\label{p:2-6-1}
	Assume that $t\geq 0$  and \eqref{S32} holds. Then $I_{K,t,f}^{\ga}$ has a second critical point $v_t\neq u_t$. In particular, $v_t$ solves \eqref{P}.
\end{proposition}

\begin{proof}
	
	Let $t\geq 0$  and let $u_t$ be the critical point of $I_{K,t,f}^{\ga}$ obtained in Proposition~\ref{p:30-7-1}.
Let $W_{\ga,t}$ be a positive radial  ground state solution of~\eqref{Q}. Set $w_{\ga, t}^{\tau}(x):=W_{\ga,t}(\frac{x}{\tau})$. Let $\x\in\Rn$ be such that $K(\x)=\|K\|_{L^{\infty}(\Rn)}.$\smallskip

\noindent\underline{\bf Claim 1:} $u_t+K(\x)^{-\frac{N-2s}{4s}}\w\in\Sigma_2^t$ for $\tau >0$ large enough.
	
Indeed, as $\|K\|_{L^{\infty}(\Rn)}\geq 1$, $0\leq t<2s$ and $u_t,\;\w>0,$ using Cauchy's inequality, with $\varepsilon>0$, we have
\begin{align*}
\psi_t\Big(u_t+K(\x)^{-\frac{N-2s}{4s}}\w\Big) &= \|u_t+K(\x)^{-\frac{N-2s}{4s}}\w\|_{\ga}^2\\ &-\big(2_s^*(t)-1\big)K(\x)\int_{\Rn}\frac{|u_t+K(\x)^{-\frac{N-2s}{4s}}\w|^{2_s^*(t)}}{|x|^t}\;{\rm d}x\\ &\leq\|u_t\|_{\ga}^2+K(\x)^{-\frac{N-2s}{2s}}\|\w\|_{\ga}^2+2K(\x)^{-\frac{N-2s}{4s}}\langle u_t,\w\rangle_{\ga}\\
& -\big(2_s^*(t)-1\big)
\bigg\{\int_{\Rn}\frac{|u_t|^{2_s^*(t)}}{|x|^t}\;{\rm d}x + K(\x)^{-\frac{N-t}{2s}}\int_{\Rn}\frac{|\w|^{2_s^*(t)}}{|x|^t}\;{\rm d}x\bigg\}\\
&\leq\|u_t\|_{\ga}^2+K(\x)^{-\frac{N-2s}{2s}}\|\w\|_{\ga}^2+2K(\x)^{-\frac{N-2s}{4s}}\big(\frac{\varepsilon}{2}\|\w\|_{\ga}^2+\frac{1}{2\varepsilon}\|u_t\|_{\ga}^2\big)\\
&\;\; -\big(2_s^*(t)-1\big)\bigg\{\int_{\Rn}\!\!
\frac{|u_t|^{2_s^*(t)}}{|x|^t}{\rm d}x + K(\x)^{-\frac{N-t}{2s}}\tau^{N-t}\int_{\Rn}
\!\!\frac{|W_{\ga,t}|^{2_s^*(t)}}{|x|^t}{\rm d}x\bigg\}\\
&=\big(1+\frac{1}{\varepsilon}\big)\|u_t\|^2_{\Hs}-(2^*_s(t)-1) \|u_t\|^{2^*_s(t)}_{L^{2^*_s(t)}(\Rn, |x|^{-t})}\\
&\;\;+\|W_{\ga,t}\|_{\ga}^2\bigg[(1+\var)\tau^{N-2s}-\big(2_s^*(t)-1\big)K(\x)^{-\frac{N-t}{2s}}\tau^{N-t}\bigg]\\
&<0\quad\mbox{ for\,  $\tau>0$ large enough. }
\end{align*}
Therefore, $u_t+K(\x)^{-\frac{N-2s}{4s}}\w\in\Sigma_2^t$ for $\tau>0$ large enough. Hence, Claim $1$ follows.\smallskip
	
\noindent\underline{\bf Claim 2:} $I_{K,t,f}^{\ga}\Big(u_t+K(\x)^{-\frac{N-2s}{4s}}\w\Big)<I_{K,t,f}^{\ga}(u_t)+I_{1,t,0}^{\ga}\Big(K(\x)^{-\frac{N-2s}{4s}}\w\Big)$ for all $\tau>0.$\\
	Indeed, as $u_t,\;\w>0$ taking $K(\x)^{-\frac{N-2s}{4s}}\w$ as the test function in \eqref{P}, we get
\begin{equation}\label{S33}\begin{aligned}
\langle u_t,K(\x)^{-\frac{N-2s}{4s}}\w\rangle_{\ga} = K(\x)^{-\frac{N-2s}{4s}}\int_{\Rn}&K(x)\frac{u_t^{2_s^*(t)-1}\w}{|x|^t} {\rm d}x\\
&+\prescript{}{\hms}{\langle}f,K(\x)^{-\frac{N-2s}{4s}}\w{\rangle}_{\dot{H}^s}.
\end{aligned}\end{equation}
Therefore, using the above equality together with the fact that $K\geq 1$ yields
\begin{align*}
I_{K,t,f}^{\ga}\Big(u_t+K(\x)^{-\frac{N-2s}{4s}}\w\Big)&= \frac{1}{2}\|u_t\|_{\ga}^2+\frac{1}{2}K(\x)^{-\frac{N-2s}{2s}}\|\w\|_{\ga}^2+K(\x)^{-\frac{N-2s}{4s}}\langle u_t,\w\rangle_{\ga}\\
&\;\;\;-\frac{1}{2_s^*(t)}\int_{\Rn}K(x)\frac{\Big(u_t+K(\x)^{-\frac{N-2s}{4s}}\w\Big)^{2_s^*(t)}}{|x|^t}\;{\rm d}x\\ &\;\;-\prescript{}{\hms}{\langle}f,u_t{\rangle}_{H^s}-K(\x)^{-\frac{N-2s}{4s}}\prescript{}{\hms}{\langle}f,\w{\rangle}_{\dot{H}^s}\\
&= I_{K,t,f}^{\ga}(u_t)+I_{1,t,0}^{\ga}\Big(K(\x)^{-\frac{N-2s}{4s}}\w\Big)+K(\x)^{-\frac{N-2s}{4s}}\langle u_t,\w\rangle_{\ga}\\
&\;\;+\frac{1}{2_s^*(t)}\int_{\Rn}K(x)\frac{u_t^{2_s^*(t)}}{|x|^t}\;{\rm d}x+\frac{K(\x)^{-\frac{N-t}{2s}}}{2_s^*(t)}\int_{\Rn}\frac{(\w) ^{2_s^*(t)}}{|x|^t}\;{\rm d}x\\
&\;\;-\frac{1}{2_s^*(t)}\int_{\Rn}K(x)\frac{\Big(u_t+K(\x)^{-\frac{N-2s}{4s}}\w\Big)^{2_s^*(t)}}{|x|^t}\;{\rm d}x\\
&\;\;-K(\x)^{-\frac{N-2s}{4s}}\prescript{}{\hms}{\langle}f,\w{\rangle}_{\dot{H}^s}\\
&\leq I_{K,t,f}^{\ga}(u_t)+I_{1,t,0}^{\ga}\Big(K(\x)^{-\frac{N-2s}{4s}}\w\Big)\\
&\;\;+K(\x)^{-\frac{N-2s}{4s}}\int_{\Rn}K(x)\frac{u_t^{2_s^*(t)-1}\w}{|x|^t}\;{\rm d}x +\frac{1}{2_s^*(t)}\int_{\Rn}K(x)\frac{u_t^{2_s^*(t)}}{|x|^t}\;{\rm d}x\\
&\;\;+\frac{K(\x)^{-\frac{N-t}{2s}}}{2_s^*(t)}\int_{\Rn}\frac{(\w) ^{2_s^*(t)}}{|x|^t}\;{\rm d}x\\
&\;\;-\frac{1}{2_s^*(t)}\int_{\Rn}K(x)\frac{\Big(u_t+K(\x)^{-\frac{N-2s}{4s}}\w\Big)^{2_s^*(t)}}{|x|^t}\;{\rm d}x\\
&\leq I_{K,t,f}^{\ga}(u_t)+I_{1,t,0}^{\ga}\Big(K(\x)^{-\frac{N-2s}{4s}}\w\Big)\\
&\;\;+\frac{1}{2_s^*(t)}\int_{\Rn}K(x)\bigg[2_s^*(t)K(\x)^{-\frac{N-2s}{4s}}\frac{u_t^{2_s^*(t)-1}\w}{|x|^t}\\
&\;\;+\frac{u_t^{2_s^*(t)}}{|x|^t}+K(\x)^{-\frac{N-t}{2s}}\frac{(\w)^{2_s^*(t)}}{|x|^t}-\frac{\Big(u_t+K(\x)^{-\frac{N-2s}{4s}}\w\Big)^{2_s^*(t)}}{|x|^t}\bigg]\;{\rm d}x\\
&<I_{K,t,f}^{\ga}(u_t)+I_{1,t,0}^{\ga}\Big(K(\x)^{-\frac{N-2s}{4s}}\w\Big).
\end{align*}
Hence the claim follows.
As $$\|\w\|_{\ga}^2=\tau^{N-2s}\|W_{\ga,t}\|_{\ga}^2, \quad \|\w\|_{L^{2^*_s(t)}(\Rn, |x|^{-t})}^{2^*_s(t)}=\tau^{N}\|W_{\ga,t}\|_{\ga}^2,$$ and $0\leq t<2s<N$, it is easy to see using the definition of
$I_{1,t,0}^{\ga}\Big(K(\x)^{-\frac{N-2s}{4s}}\w\Big)$ that
\be\label{S34}
	\lim_{\tau\to \infty}	I_{1,t,0}^{\ga}\Big(K(\x)^{-\frac{N-2s}{4s}}\w\Big) = -\infty
	\ee
Consequently, a straight forward computation yields that
$$\sup_{\tau>0}I_{1,t,0}^{\ga}\Big(K(x_0)^{-\frac{N-2s}{4s}}\w\Big) = I_{1,t,0}^{\ga}\Big(K(\x)^{-\frac{N-2s}{4s}}w_{\ga,t}^{\tau_{\max}}\Big), \quad\mbox{where}\,\, \tau_{\max}=K(\x)^{\tfrac{1}{2s}}.$$
Therefore, substituting the value of $\tau_{\max}$ in the definition of $I_{1,t,0}^\ga$, it is not difficult to check that
$$\sup_{\tau>0}I_{1,t,0}^{\ga}\Big(K(x_0)^{-\frac{N-2s}{4s}}\w\Big)=I_{1,t,0}^{\ga}(W_{\ga,t}).$$
Combining the above relation with Claim $2$ and \eqref{S34}, we obtain
\begin{gather}\label{S35}
  I_{K,t,f}^{\ga}\Big(u_t+K(\x)^{-\frac{N-2s}{4s}}\w\Big)<I_{K,t,f}^{\ga}(u_t)+I_{1,t,0}^{\ga}(W_{\ga,t})\quad \mbox{ for all }\tau>0,\\
\label{S36'}
I_{K,t,f}^{\ga}\Big(u_t+K(\x)^{-\frac{N-2s}{4s}}\w\Big)<I_{K,t,f}^{\ga}(u_t)\quad \mbox{ for } \tau  \mbox{ large enough}.
\end{gather}
Now, fix $\tau_0>0$ large enough such that Claim $1$ and \eqref{S36'} are satisfied. Set
 $$\kappa_t:=\inf_{\theta\in\Theta_t}\max_{r\in[0,1]}
 I_{K,t,f}^{\ga}\Big(\theta(r)\Big),$$
where $\Theta_t :=\bigg\{\theta\in C\Big([0,1],\Hs\Big)\;:\;\theta(0)=u_t,\;\,\theta(1)
=u_t+K(\x)^{-\frac{N-2s}{4s}}w^{\tau_0}_{\ga,t}\bigg\}$.
As $u_t\in\Sigma_1^t$, $u_t+K(\x)^{-\frac{N-2s}{4s}}w^{\tau_0}_{\ga,t}\in\Sigma_2^t$ for every $\theta\in \Theta_t$, there exists  $r_{\theta}\in (0,1)$ such that $\theta(r_{\theta})\in\Sigma^t$. Thus
	$$\max_{r\in[0,1]}I_{K,t,f}^{\ga}(\theta(r))\geq I_{K,t,f}^{\ga}(\theta(r_{\theta}))\geq \inf_{u\in\Sigma^t}I_{K,t,f}^{\ga}(u)=c_1^t.$$
Hence, $$\kappa_t \geq c_1^t>c_0^t=I_{K,t,f}^{\ga}(u_t).$$
	Here in the last inequality we have used Lemma~\ref{l:32}.
\smallskip
	
\noindent\underline{\bf Claim 3:} $\kappa_t<I_{K,t,f}^{\ga}(u_t)+I_{1,t,0}^{\ga}(W_{\ga,t}).$\\
	Note that $\lim_{\tau\to 0}\|\w\|_{\ga}=0,$  thus if we define $\bar{\theta}(r):= u_t+K(\x)^{-\frac{N-2s}{4s}}w^{r\tau_0}_{\ga,t},$ then $\bar{\theta}\in\Theta_t$ and
	$\lim_{r\to 0}\|\bar{\theta}(r)-u_t\|_{\ga}=0.$ Therefore by \eqref{S35},
	\be
	\kappa_t \leq \max_{r\in[0,1]}I_{K,t,f}^{\ga}\Big(\bar{\theta}(r)\Big) = \max_{r\in[0,1]}I_{K,t,f}^{\ga}\Big(u_t+K(\x)^{-\frac{N-2s}{4s}}w_{\ga,t}^{r\tau_0}\Big)<I_{K,t,f}^{\ga}(u_t)+I_{1,t,0}^{\ga}(W_{\ga,t}),\no
	\ee
that is,
\be\lab{4-6-1} I_{K,t,f}^{\ga}(u_t)<\kappa_t<I_{K,t,f}^{\ga}(u_t)+I_{1,t,0}^{\ga}(W_{\ga,t}) \quad \mbox{for all }t\geq 0.\ee
	Using Ekeland's variational principle, there exists a $(PS)$ sequence $(v_n^t)_n$ of $I_{K,t,f}^{\ga}$ at level $\kappa_t$ for all $t\geq 0$. Since any $(PS)$ for $I_{K,t,f}^{\ga}$ is bounded and $\kappa_t <I_{K,t,f}^{\ga}(u_t)+I_{1,t,0}^{\ga}(W_{\ga,t})$, using Theorem~\ref{th:PS}, in the case of $t>0$, there exists $v_t\in\Hs$ such that $v_n^t\to v_t$ in $\Hs$, with $I_{K,t,f}^{\ga}(v_t)  =\kappa_t$ and $(I_{K,t,f}^{\ga})'(v_t)=0$. Moreover, $I_{K,t,f}^{\ga}(u_t)<\kappa_t$
implies that $u_t\neq v_t$. Hence we have proved the proposition for $t>0$.
\vspace{2mm}

Next let us assume that $t=0$ so that we are in case $(ii)$ of Theorem~\ref{th:ex-f}
and so
\be\lab{4-6-3}
K(\x)=\|K\|_{L^\infty(\mathbb R^N)}<\bigg(\frac{S}{S_{\ga,0,s}}\bigg)^\frac{N}{N-2s}
\ee
holds by assumption.
Let $W$ denote the unique positive solution of \eqref{W}. As $W_{\ga,0}$ is a minimum energy positive solution (ground state solution) of \eqref{Q}, with $t=0$, it follows that
$$I_{1,0,0}^{\ga}(W_{\ga,0})\leq I_{1,0,0}^{\ga}(W)<I_{1,0,0}^{0}(W),$$
where the last inequality is due to the fact that $W>0$ and so $\int_{\Rn}\frac{|W|^2}{|x|^{2s}}{\rm d}x>0$.
 Since $S$ and $S_{\ga,0,s}$ are achieved by $W$ and $W_{\ga,0}$ respectively, it is easy to see that
$\|W\|^2_{\Hs}=S^\frac{N}{2s}$ and $\|W_{\ga,0}\|^2_{\ga}=S_{\ga,0,s}^\frac{N}{2s}$. On the other hand, as $I_{1,0,0}^{0}(W)=\frac{s}{N}\|W\|^2_{\dot{H}^s}$
and $I_{1,0,0}^{\ga}(W_{\ga,0})=\frac{s}{N}\|W_{\ga,0}\|^2_{\ga}$, we obtain
$\frac{I_{1,0,0}^{\ga}(W_{\ga,0})}{I_{1,0,0}^{0}(W)}=\big(\frac{S_{\ga,0,s}}{S}\big)^\frac{N}{2s}$.
This  together with \eqref{4-6-3} yields
$$I_{1,0,0}^{\ga}(W_{\ga,0})<K(\x)^{-\frac{N-2s}{2s}}I_{1,0,0}^{0}(W)\leq K(x)^{-\frac{N-2s}{2s}}I_{1,0,0}^{0}(W) \quad\mbox{for all }x\in\Rn.$$
Combining the above inequality with \eqref{4-6-1} yields
$$I_{K,0,f}^{\ga}(u_0)<\kappa_0<\min\bigg\{I_{K,0,f}^{\ga}(u_0)+K(x)^{-\frac{N-2s}{2s}}I_{1,0,0}^{0}(W), \quad I_{K,0,f}^{\ga}(u_0)+I_{1,0,0}^{\ga}(W_{\ga,0})\bigg\}.$$	
Hence, again using Theorem~\ref{th:PS} (as in the case $t>0$), we can conclude that the $(PS)$ sequence $(v_n^0)_n$ converges strongly to some $v_0\in \Hs$, with 	$I_{K,0,f}^{\ga}(v_0)  =\kappa_0$ and $(I_{K,0,f}^{\ga})'(v_0)=0$. As before, $I_{K,0,f}^{\ga}(u_0)<\kappa_0$ implies that $u_0\neq v_0$. Hence we have completed the proof for all $t\geq 0$.
\end{proof}

\begin{lemma}\lab{l:J1.3}
	If $\|f\|_{\hms}< C_t\sqrt{1-\frac{\ga}{\ga_{N,s}}}S_{\ga,t,s}^{\tfrac{N-t}{4s-2t}}$, then \eqref{S32} holds.
\end{lemma}

\begin{proof}
	By the given assumption, there exists $\varepsilon>0$ such that
	\be
	\|f\|_{\hms}< C_t\sqrt{1-\frac{\ga}{\ga_{N,s}}}S_{\ga,t,s}^{\tfrac{N-t}{4s-2t}}-\varepsilon.\no
	\ee
	Combining this with Lemma~\ref{l31}, for all $u^t\in \Si^t$, it holds
\begin{align*}
	\prescript{}{\hms}{\langle}f,u^t{\rangle}_{\dot{H}^s} \leq \|f\|_{\hms}\|u^t\|_{\Hs}
	&\leq \bigg(1-\frac{\ga}{\ga_{N,s}}\bigg)^{-\frac{1}{2}} \|f\|_{\hms}\|u^t\|_{\ga}\\
	&<C_t S_{\ga,t,s}^{\frac{N-t}{4s-2t}}\|u^t\|_{\ga}-\varepsilon\bigg(1-\frac{\ga}{\ga_{N,s}}\bigg)^{-\frac{1}{2}}\|u^t\|_{\ga}\\
	&\leq \frac{4s-2t}{N-2t+2s}\|u^t\|_{\ga}^2-\varepsilon\bigg(1-\frac{\ga}{\ga_{N,s}}\bigg)^{-\frac{1}{2}}\|u^t\|_{\ga}.
	\end{align*}
	Hence,
$$	\inf_{u\in\Sigma^t}\[\frac{4s-2t}{N-2t+2s}\|u\|_{\ga}^2-\prescript{}{\hms}{\langle}f,u{\rangle}_{\dot{H}^s}\]\geq \varepsilon \bigg(1-\frac{\ga}{\ga_{N,s}}\bigg)^{-\frac{1}{2}}\inf_{u\in\Sigma^t}\|u\|_{\ga}.
$$
Since  $\|u\|_{\ga}$ is bounded away from $0$ on $\Sigma^t$
by Remark~\ref{r:30-7-1}, the above expression implies that
\be\label{S311}
\inf_{u\in\Sigma^t}\[\frac{4s-2t}{N-2t+2s}\|u\|_{\ga}^2
-\prescript{}{\hms}{\langle}f,u{\rangle}_{\dot{H}^s}\]>0.
\ee
On the other hand,
\begin{align}\label{S312}
\eqref{S32} &\iff C_t\frac{\|u\|_{\ga}^{\frac{N-2t+2s}{2s-t}}}
{\bigg(\displaystyle\int_{\Rn}\frac{|u|^{2_s^*(t)}}{|x|^t}\;{\rm d}x\bigg)^{\frac{N-2s}{4s-2t}}}
-\prescript{}{\hms}{\langle}f,u{\rangle}_{\dot{H}^s}>0\quad\mbox{for } \|u\|_{L^{2_s^*(t)}(\Rn, |x|^{-t})}=1\no\\
	&\iff C_t\frac{\|u\|_{\ga}^{\frac{N-2t+2s}{2s-t}}}{\bigg(\displaystyle\int_{\Rn}\frac{|u|^{2_s^*(t)}}{|x|^t}\;{\rm d}x\bigg)^{\frac{N-2s}{4s-2t}}}-\prescript{}{\hms}{\langle}f,u{\rangle}_{\dot{H}^s}>0
\quad\mbox{for }u\in\Sigma^t\\
	&\iff \frac{4s-2t}{N-2t+2s}\|u\|_{\ga}^2-\prescript{}{\hms}{\langle}f,u{\rangle}_{\dot{H}^s}>0
\quad\mbox{for }u\in\Sigma^t.\no
\end{align}
Clearly, \eqref{S311} ensures the RHS of \eqref{S312} holds. The lemma now follows.
\end{proof}
\vspace{2mm}

\noindent{\it Proof of Theorem~$\ref{th:ex-f}$ completed.}
Combining Propositions~\ref{p:30-7-1} and~\ref{p:2-6-1} with Lemma~\ref{l:J1.3}, we conclude the proof of Theorem~\ref{th:ex-f}.
\qed

\appendix
\section{}\label{app}
\setcounter{equation}{0}

\begin{lemma}\label{L1}
Let $(v_n)_n\subseteq \Hs$ be a $(PS)$ sequence for $\bar I_{K,0,0}^\ga$  at the level $d$. Assume that, there exist sequences $(y_n)_n\to y\in\Rn$, $r_n\to 0\in \R^+\cup\{0\}$ such that $w_n(x)= r_n^{\frac{N-2s}{2}}v_n(r_nx+y_n)$ converges weakly in $\Hs$ and a.e. to some $w\in\Hs$. If $\frac{|y_n|}{r_n}\to \infty$, then $K(y)^{\frac{N-2s}{4s}}w$ solves \eqref{W}.
Moreover, $$z_n:= v_n - r_n^{-\frac{N-2s}{2}}w(\frac{x-y_n}{r_n})$$ is a $(PS)$ sequence for $\bar I_{K,0,0}^\ga$ at the level $d- K(y)^{-\frac{N-2s}{2s}}\bar{I}_{1,0,0}^{0}(K(y)^{\frac{N-2s}{4s}}w)$.
\end{lemma}
	
\begin{proof} Let $(v_n)_n\subseteq \Hs$ be a $(PS)$ sequence for $\bar I_{K,0,0}^\ga$  at the level $d$ and $\phi$ be an arbitrary $C^\infty_c(\Rn)$ function.
Put
$\phi_n(x):=r_n^{-\frac{N-2s}{2}}\phi(\tfrac{x-y_n}{r_n})$. Thus,
\begin{align}\lab{1-6-1}
\langle w,\phi\rangle_{\dot{H}^s} &= \lim_{n\to\infty}\langle w_n,\phi\rangle_{\dot{H}^s}\no\\
&=\lim_{n\to\infty}\frac{C_{N,s}}{2}\iint_{\R^{2N}}\frac{(w_n(x)-w_n(y))(\phi(x)-\phi(y))}{|x-y|^{N+2s}}\;{\rm d}x{\rm d}y\no\\
&=\lim_{n\to\infty}\frac{C_{N,s}}{2}\iint_{\R^{2N}}\frac{(v_n(x)-v_n(y))(\phi_n(x)-\phi_n(y))}{|x-y|^{N+2s}}\;{\rm d}x{\rm d}y\\
		&=\lim_{n\to\infty}\ga\int_{\Rn}\frac{v_n\phi_n}{|x|^{2s}}\;{\rm d}x+\int_{\Rn}K(x)|v_n|^{2_s^*-2}v_n\phi_n\;{\rm d}x\no\\
		&=\lim_{n\to\infty}\bigg[\ga \int_{\Rn}\frac{w_n\phi}{|x+r_n^{-1}y_n|^{2s}}\;{\rm d}x +\int_{\Rn} K(r_nx+y_n)|w_n|^{2_s^*-2}w_n\phi\; {\rm d}x \bigg].\no
\end{align}
Since $r_n^{-1}|y_n|\to\infty$, for each fixed $\phi$ we have
$$\lim_{n\to\infty}\int_{\Rn}\frac{w_n\phi}{|x+r_n^{-1}y_n|^{2s}}\;{\rm d}x =0.$$
Therefore, taking the limit as $n\to\infty$ in \eqref{1-6-1}, we obtain	$(-\De)^sw=K(y)|w|^{2_s^*-2}w$, or equivalently $K(y)^{\frac{N-2s}{4s}}w$ solves \eqref{W}.
Moreover,
$$\int_{\Rn}\frac{w_nw}{|x-r_n^{-1}y_n|^{2s}}{\rm d}x=\int_{\Rn}\frac{|w|^2}{|x-r_n^{-1}y_n|^{2s}}{\rm d}x=o(1).$$
Therefore, proceeding as  in Claim~2 of Step~5 in the proof of Theorem~\ref{th:PS}, we obtain as $n\to\infty$
$$\bar I_{K,0,0}^\ga(z_n)=\bar I_{K,0,0}^\ga(v_n)-K(y)^{-\frac{N-2s}{2s}}\bar I_{1,0,0}^{0}(K(y)^{\frac{N-2s}{4s}}w)+o(1).$$
To prove that $\prescript{}{\hms}{\big\langle}\bar I_{K,0,0}^\ga(z_n), \va{\big\rangle}_{\dot H^s}=o(\|\va\|)$, we proceed as in the proof of
Claim~2 of Step~5 in Theorem~\ref{th:PS}, the only additional estimate we need to check is
$$\int_{\Rn}\frac{w\va_n}{|x-r_n^{-1}y_n|^{2s}}{\rm d}x=o(\|\va_n\|),$$ 	
where $\va_n=r_n^\frac{N-2s}{2s}\va(r_nx+y_n)$, $\|\va_n\|=\|\va\|$. This estimate follows from the Cauchy-Schwartz and the H\"{o}lder inequalities.
\end{proof}

\begin{lemma}\label{L2}
	Let $\mathcal{K}$ denote the Kelvin transform in $\Rn.$ If $(r_n)_n\subset \R^+\cup\{0\}$ and $(y_n)_n\subset \Rn$ are sequences such that $\tfrac{|y_n|}{r_n}\to\infty$
and $W$ is a positive solution of~\eqref{W}, then in the sense of $\Hs$-norm as $n\to\infty$
\be\label{PSL1}
r_n^{-\frac{N-2s}{2}}\mathcal{K}\bigg(W(\tfrac{x-y_n}{r_n})\bigg) = \big(r_n|y_n|^{-2}\big)^{-\frac{N-2s}{2}}
W\bigg(\frac{x-\tfrac{y_n}{|y_n|^2}}{r_n|y_n|^{-2}}\bigg)+o(1).
\ee
\end{lemma}

\begin{proof} Let the assumptions and notation of ghe statement hold.
Let $W$ be a positive solution of~\eqref{W}. Then   $W(x)=C_{N,s}(1+|x|^2)^{-\frac{N-2s}{2}}$ thanks to \cite{CLO}.
The $\Hs$ norm is invariant under the scaling so that
\be\label{PSL2}
v\mapsto \tilde{v}(x):= \bigg(\frac{r_n}{|y_n|^{2}}\bigg)^{\frac{N-2s}{2}}v\bigg(\frac{r_n}{|y_n|^{2}}x+\frac{y_n}{|y_n|^2}\bigg),
\ee
we can apply it to each side of $\eqref{PSL1}$ to check the convergence. The RHS of $\eqref{PSL1}$ becomes $W+o(1)$. The LHS of $\eqref{PSL1}$, after some algebraic computation, is transformed into
	$$W^n(x):=C_{N,s}\bigg(1+\frac{r_n}{|y_n|}\langle x,y_n\rangle+\big(1+r_n^2|y_n|^{-2}\big)|x|^2)\bigg)^{-\frac{N-2s}{2}}.$$
	As $\frac{r_n}{|y_n|}\to 0$, clearly $W^n\to W$ in $\Hs$. Hence the proof is complete.
\end{proof}	

{\bf Acknowledgement}:  The research of M.~Bhakta is partially supported by the {\em SERB MATRICS grant (MTR/2017/000168) and WEA grant (WEA/2020/000005)}. S.~Chakraborty is partially supported by {\em NBHM grant 0203/11/2017/RD-II.}

P. Pucci is member of the {\em Gruppo Nazionale per l'Analisi Ma\-te\-ma\-ti\-ca, la Probabilit\`a e le loro Applicazioni}
(GNAMPA) of the {\em Istituto Nazionale di Alta Matematica} (INdAM)
and is partly supported by the   INdAM -- GNAMPA Project
{\em Equazioni alle derivate parziali: problemi e mo\-del\-li} (Prot\_U-UFMBAZ-2020-000761).
P. Pucci was also partly supported by of the {\em Fondo Ricerca di Base di Ateneo-Eser\-ci\-zio 2017--2019} of the University of Perugia, named {\em PDEs and Nonlinear Analysis}.


\medskip


\begin{thebibliography}{99}
\bibitem{AADP}{\sc Abdellaoui, B.; Attar, A.; Dieb, A.; Peral, I.}, Attainability of the fractional Hardy constant with nonlocal mixed boundary conditions: applications, Discrete Contin. Dyn. Syst. 38(12) (2018) 5963--5991.

\bibitem{AMPP}{\sc  Abdellaoui, B.; Medina, M.; Peral, I.; Primo, A.}, {\em The effect of the {H}ardy potential in some {C}alder\'{o}n-{Z}ygmund properties for the fractional {L}aplacian}, {J. Differential Equations},
{260} 
(2016), {8160--8206}.

\bibitem{AM}{\sc  Adimurthi; Mallick, A.} {\em A Hardy type inequality on fractional order Sobolev spaces on the Heisenberg group}. {Ann. Sc. Norm. Super. Pisa Cl. Sci. (5)} 18 (2018), no. 3, 917--949. 

\bibitem{AT}{\sc  Alarc\'{o}n, S.; Tan, J.}, {\em Sign-changing solutions for some nonhomogeneous nonlocal critical elliptic problems}, Discrete Contin. Dyn. Syst. 39 (2019), no. 10, 5825--5846.

\bibitem{BBGM}{\sc Bhakta, M; Biswas, A.; Ganguly, D.; Montoro, L.}, {\em Integral representation of solutions using Green function for fractional Hardy equations}, J. Differential Equations 269 (2020), no. 7, 5573--5594.

\bibitem{BCG}{\sc  Bhakta, M.; Chakraborty, S.; Ganguly, D.}, {\em Existence and Multiplicity of positive solutions of certain nonlocal scalar field equations}, preprint, arXiv: 1910:07919.

\bibitem{BP}{\sc Bhakta, M.; Pucci, P.}, {\em On multiplicity of positive solutions for nonlocal equations with critical nonlinearity}, Nonlinear Anal. 197 (2020), 111853, 22 pp.

\bibitem{BS}{\sc Bhakta, M.; Sandeep, K.}, {\em Hardy-Sobolev-Maz'ya type equations in bounded domains}, J. Differential Equations 247 (2009), no. 1, 119--139.

\bibitem{GGJP} {\sc Bogdan, K.; Grzywny, T.; Jakubowski, T.; Pilarczyk, D.},  {\em Fractional Laplacian with Hardy Potential},  Comm. Partial Differential Equations  44 
 (2019), 20--50.


\bibitem{CLO}{\sc Chen, W.; Li, C.; Ou, B.}, {\em Classification of solutions for an integral equation,}
{ Comm. Pure Appl. Math.} 59 (2006), no. 3, 330--343.


\bibitem{DPQ}{\sc  Del Pezzo, L.M.; Quaas, A.}, {\em A Hopf's lemma and a strong minimum principle for the fractional $p$-Laplacian}, J. Differential Equations 263 (2017), no. 1, 765--778.

\bibitem{DMPS}{\sc Dipierro, S.; Montoro, L.;  Peral, I.;  Sciunzi, B.}, {\em Qualitative properties of positive solutions to nonlocal critical problems involving the Hardy-Leray potential}, {Calc. Var. Partial Differential Equations} 55 (2016),
    Art. 99, 29 pp.


\bibitem{FP}{\sc Felli, V.; Pistoia, A.}, {\em Existence of Blowing-up Solutions for a Nonlinear Elliptic Equation with Hardy Potential and Critical Growth},
Comm. Partial Differential Equations 31, (2006), no. 1--3,  21--56.

\bibitem{FLS} {\sc Frank, R. L.; Lieb, E. H.; Seiringer, R.}, {\em Hardy-Lieb-Thirring inequalities for fractional Schr\"odinger operators}, {J. Amer. Math. Soc.} 21 (2008),
    925--950.

\bibitem{FS} {\sc Frank, R. L.;  Seiringer, R.}, {\em Non-linear ground
state representations and sharp Hardy inequalities},
J. Funct. Anal. 255 (2008), 3407--3430.

\bibitem{GRSZ}{\sc Ghoussoub, N.; Robert, F.; Shakerian, S.; Zhao, M.}, {\em Mass and asymptotics associated to fractional Hardy-Schr\"{o}dinger operators in critical regimes}, Comm. Partial Differential Equations 43 (2018), no. 6, 859--892.

\bibitem{GS}{\sc Ghoussoub, N.; Shakerian, S}, {\em Borderline variational problems involving fractional Laplacians and
critical singularities}, Adv. Nonlinear Stud. 15 (2015), no. 3, 527--555.

\bibitem{Li}{\sc Lions, P.-L. }, {\em The concentration-compactness principle in the calculus of variations. The limit
case}, Rev. Mat. Iberoamericana 1 (1985), 45--121.

\bibitem{Ma} {\sc  Mallick, A.}, Extremals for fractional order Hardy-Sobolev-Maz'ya inequality, {\em Calc. Var. Partial Differential Equations} 58 (2019), no. 2, no. 45, 37 pp.


\bibitem{PS}{\sc  Palatucci, G.; Pisante, A.}, {\em Improved Sobolev embeddings, profile decomposition, and concentration-compactness for fractional Sobolev spaces}, Calc. Var. Partial Differential Equations 50 (2014), no. 3-4, 799--829.

\bibitem{PS-2}{\sc  Palatucci, G.; Pisante, A.}, {\em A global compactness type result for Palais-Smale sequences in fractional Sobolev spaces}, {Nonlinear Anal.} 117 (2015), 1--7.

\bibitem{RS}{\sc Ros-Oton, X.;  Serra, J.}, {\em The Dirichlet problem for the fractional Laplacian: regularity up to the boundary}, J. Math. Pures Appl, 
    101 (2014), 275--302.

\bibitem{SZY}{\sc  Shang, X.; Zhang, J.; Yang, Y.}, {\em Positive solutions of nonhomogeneous fractional Laplacian problem with critical exponent,} {Commun. Pure Appl. Anal.} 13 (2014), no. 2, 567--584.

\bibitem{Sm}{\sc Smets, D.}, {\em Nonlinear Schr\"odinger equations with Hardy potential and critical nonlinearities}, {Trans. Amer. Math. Soc.} 357 (2005), no. 7, 2909--2938.

\bibitem{St}{\sc Struwe, M.},  Variational methods. Applications to nonlinear partial differential equations and Hamiltonian systems, Fourth edition, Ergebnisse der Mathematik und ihrer Grenzgebiete, 3. Folge. A Series of Modern Surveys in Mathematics 34, Springer-Verlag, Berlin, 2008, xx+302 pp.


\bibitem{TF}{\sc Tintarev, K., Fieseler, K.-H.} {\em Concentration Compactness. Functional-Analytic Grounds and Applications}.
Imperial College Press, London 2007, xii+264 pp.

\bibitem{WZ}{\sc Wang, F.; Zhang, Y.}, {\em Existence of multiple positive solutions for nonhomogeneous fractional Laplace problems with critical growth}, Bound. Value Probl. 2019, Paper no. 169, 21 pp.
\end{thebibliography}
\end{document}